\documentclass[12pt]{article}
\usepackage{amsmath}
\usepackage{graphicx,psfrag,epsf}
\usepackage{enumerate}
\usepackage{natbib}
\usepackage{booktabs} % For professional table rules
\usepackage{caption}  % For better caption spacing

\usepackage{amsmath}
\usepackage{amssymb}
\usepackage{mathtools}
\usepackage{amsthm}

\usepackage[utf8]{inputenc} % allow utf-8 input
\usepackage[T1]{fontenc}    % use 8-bit T1 fonts
\usepackage{hyperref}       % hyperlinks
\usepackage{url}            % simple URL typesetting
\usepackage{booktabs}       % professional-quality tables
\usepackage{amsfonts}       % blackboard math symbols
\usepackage{nicefrac}       % compact symbols for 1/2, etc.
\usepackage{microtype}      % microtypography
\usepackage{xcolor}         % colors

\usepackage{booktabs}
\usepackage{subfigure,epsfig,amsfonts}
\usepackage{amsmath}
\usepackage{amssymb}
\usepackage{amsthm}

\usepackage{caption}
\usepackage{enumerate}
\usepackage{url} % not crucial - just used below for the URL 
\usepackage{float}
\usepackage{framed}
\usepackage{algorithm}
\usepackage{algorithmic}
\usepackage{rotating}

\usepackage{xr}

\usepackage{graphicx,psfrag,epsf}
\usepackage{enumerate}
\usepackage{url} % not crucial - just used below for the URL 
\usepackage{float}
\usepackage{algorithm,algorithmic}
\usepackage[algo2e]{algorithm2e}
\usepackage{multirow}
\usepackage{lipsum}
\usepackage{amsfonts}

\usepackage{graphicx}   % Written by David Carlisle and Sebastian Rahtz
\usepackage{amsmath}
\usepackage{amsthm}
\usepackage{amsfonts}
\usepackage{amssymb}
\usepackage{graphicx}
\usepackage{subfigure}
\usepackage{transparent}
\usepackage{multirow}
\usepackage{color}
\usepackage{blkarray}
\usepackage{multirow}

\usepackage{lipsum}

\newcommand{\bm}[1]{\boldsymbol{#1}}

\renewcommand{\d}{\mathrm{d}}
\renewcommand{\P}{\mathbb{P}}

\newcommand{\Y}{\bm{Y}}

\newcommand{\mE}{\mathcal{E}}

\newcommand{\mT}{\mathcal{T}}
\newcommand{\wt}[1]{\widetilde{#1}}

\newcommand{\norm}[1]{\left\Vert#1\right\Vert}
\newcommand{\abs}[1]{\left\vert#1\right\vert}

\newcommand{\wh}[1]{\smash{\widehat{#1}}}

\usepackage{array}
\newcolumntype{C}{@{\extracolsep{0.5in}}c@{\extracolsep{0pt}}}

\def\C {\,|\:}

\def\C {\,|\:}

\def\mF{\mathcal{F}}

\def\B{\bm{B}}
\def\b{\bm{\beta}}
\def\Y{\bm{Y}}

\def\X{\bm{X}}
\def\x{\bm{x}}
\def\y{\bm{y}}

\def\b{\bm{\beta}}

\renewcommand{\d}{\mathrm{d}\,}

\renewcommand{\d}{\mathrm{d}}
\renewcommand{\P}{\mathbb{P}}

\def\C{\mbox{\boldmath$C$}}

\def\x{\mbox{\boldmath$x$}}
\def\y{\mbox{\boldmath$y$}}
\def\b{\mbox{\boldmath$b$}}

\def\d{\mbox{\boldmath$d$}}

\def\B{\mbox{\boldmath$B$}}
\def\X{\mbox{\boldmath$X$}}

\def\bOmega{\mbox{\boldmath$\Omega$}}

\renewcommand{\Y}{\bm{Y}}
\usepackage{array}

\def\C {\,|\:}

\def\C {\,|\:}

\def\B{\bm{B}}
\def\b{\bm{\beta}}

\def\X{\bm{X}}
\def\x{\bm{x}}
\def\y{\bm{y}}

\def\b{\bm{\beta}}

\renewcommand{\d}{\mathrm{d}\,}

\usepackage[capitalize,noabbrev]{cleveref}

\theoremstyle{plain}
\newtheorem{theorem}{Theorem}[section]

\newtheorem{lemma}[theorem]{Lemma}

\theoremstyle{definition}
\newtheorem{definition}[theorem]{Definition}

\theoremstyle{remark}

\usepackage[textsize=tiny]{todonotes}

\usepackage[utf8]{inputenc} % allow utf-8 input
\usepackage[T1]{fontenc}    % use 8-bit T1 fonts
\usepackage{hyperref}       % hyperlinks
\usepackage{url}            % simple URL typesetting
\usepackage{booktabs}       % professional-quality tables
\usepackage{amsfonts}       % blackboard math symbols
\usepackage{nicefrac}       % compact symbols for 1/2, etc.
\usepackage{microtype}      % microtypography
\usepackage{lipsum}
\usepackage{fancyhdr}       % header
\usepackage{graphicx}       % graphics
\graphicspath{{media/}}     % organize your images and other figures under media/ folder

\newcommand{\blind}{0}

\begin{document}

\def\spacingset#1{\renewcommand{\baselinestretch}%
{#1}\small\normalsize} \spacingset{1}

%%%%%%%%%%%%%%%%%%%%%%%%%%%%%%%%%%%%%%%%%%%%%%%%%%%%%%%%%%%%%%%%%%%%%%%%%%%%%%

\title{\bf Generalized Bayesian Additive Regression Trees: Theory and Software}

\if0\blind
{
  \author{Enakshi Saha\\
    Department of Epidemiology and Biostatistics\\
    University of South Carolina, Columbia}
} \else {
  \author{Anonymous Author}
} \fi

\maketitle

\if1\blind
{
  \bigskip
  \bigskip
  \bigskip
  \begin{center}
    {\LARGE\bf Title}
\end{center}
  \medskip
} \fi

\bigskip
\begin{abstract}
Bayesian Additive Regression Trees (BART) are a powerful ensemble learning technique for modeling nonlinear regression functions. Although initially BART was proposed for predicting only continuous and binary response variables, over the years multiple extensions have emerged that are suitable for estimating a wider class of response variables (e.g. categorical and count data) in a multitude of application areas. In this paper we describe a generalized framework for Bayesian trees and their additive ensembles where the response variable comes from an exponential family distribution and hence encompasses many prominent variants of BART. We derive sufficient conditions on the response distribution, under which the posterior concentrates at a minimax rate, up to a logarithmic factor. In this regard our results provide theoretical justification for the empirical success of BART and its variants. To support practitioners, we develop a Python package, also accessible in R via reticulate, that implements GBART for a range of exponential family response variables including Poisson, Inverse Gaussian, and Gamma distributions, alongside the standard continuous regression and binary classification settings. The package provides a user-friendly interface, enabling straightforward implementation of BART models across a broad class of response distributions.
\end{abstract}

\noindent%

\spacingset{1.45}
\section{Introduction}
\label{sec:intro}

Additive ensemble of Bayesian trees  \citep{chipman1998bayesian, denison1998bayesian}, more popularly known as Bayesian additive regression trees (BART) \citep{chipman2010bart} is a flexible tool that has been extremely successful in a multitude of high dimensional classification and regression tasks. Aided by efficient software implementations, (e.g. R packages \texttt{BART} of \cite{sparapani2019bart}, \texttt{bartMachine} of \cite{bleich2014variable}, \texttt{parallel BART} of \cite{pratola2014parallel}, \texttt{XBART} of \cite{he2019xbart} and \texttt{dbarts} of \cite{dorie2025package}), BART has thrived in a wide range of application areas, including causal inference \citep{hill2011bayesian, hill2013assessing, hahn2017bayesian}, dynamic treatment regimes \citep{li2024dynamic}, interaction detection \citep{du2019interaction}, survival analysis \citep{sparapani2016nonparametric}, time series analysis \citep{taddy2011dynamic, deshpande2020vc} and variable selection \citep{bleich2014variable, linero2018bayesian, liu2018abc, liu2020variable}, to name a few. Even though BART was initially proposed for predicting univariate continuous and binary response variables, due to its flexibility and impressive performance, multiple extensions have emerged over the subsequent years, that are suitable for both univariate and multivariate prediction problems where the response variable is of a wider variety (e.g. categorical and count data \citep{murray2017log}, heteroscedastic responses \citep{pratola2019heteroscedastic}, multivariate skewed responses \citep{um2023bayesian}) and/or the target regression surface is of a constrained nature (e.g. monotone BART \citep{chipman2016high}, varying coefficient BART \citep{deshpande2020vc}, BART with targeted smoothing \citep{starling2020bart} and interval-censored survival data \citep{basak2022semiparametric} etc.).

Despite a long history of empirical success, theoretical studies on Bayesian trees and forests is a relatively new area of research. Recently emerging results along this line are geared towards providing a theoretical perspective on why these models have been so successful in a wide range of classification and regression problems. Among the initial developments, \cite{rockova2019theory} and \cite{rockova2020posterior} demonstrated that the posterior concentration rate of BART equals the minimax rate up to a logarithmic factor for various tree priors. \cite{van2017bayesian} established posterior concentration results for Bayesian dyadic trees around step functions at the minimax rate. Built on these findings, \cite{rockova2020semi} derived a semiparametric Bernstein von-Mises theorem for the BART estimator. Extensions of BART, adapted to various special function types have also been studied from a theoretical perspective: \cite{linero2017bayesian} studied a version of BART suitable for smooth function estimation; \cite{castillo2019multiscale} conducted a multiscale analysis of BART; \cite{jeong2020art} derived posterior concentration results for anisotropic functions. In this paper we study the posterior concentration rates of a generalized version of BART for exponential family responses, extending these results beyond Gaussian outcomes and beyond Hölder continuity to step functions and monotone functions, thereby supplementing this newly emerging area of research.

We formulate a Generalized BART (G-BART) model that extends the existing theoretical developments in several directions. Firstly while existing results focus on Gaussian response variables, we allow the response to come from an exponential family distribution. Hence G-BART can be regarded as nonparametric extensions of the widely popular `Generalized Linear Models' (GLM) \citep{nelder1972generalized}. Many prominent Bayesian CART and BART models used in practice \citep{denison1998bayesian, chipman2010bart, murray2017log}, including the traditional BART model \citep{chipman2010bart}, can be viewed as a special case of this generalized extension. Therefore theoretical properties of these conventional adaptations of BART can be studied as direct corollaries of analogous properties for the G-BART model. Recently, \cite{linero2025generalized} proposed a generalized BART framework for exponential family responses via reversible jump MCMC. While that work focuses on computational methodology, the present paper provides the theoretical foundation for such models by establishing posterior concentration rates.

Secondly, existing results \citep{rockova2020posterior, rockova2019theory, linero2017bayesian} build upon the assumption that the underlying regression function is H\"{o}lder continuous. However given the efficacy of BART models in a multitude of prediction problems with varying degrees of complexity, the assumption of H\"{o}lder continuity seems too restrictive. In this paper we demonstrate that similar posterior optimality results can be obtained for non-smooth functions such as step functions and monotone functions in the exponential family setting, thus extending the theoretical findings on BART beyond the assumption of H\"older continuity. 

Finally, the BART model \cite{chipman2010bart} approximates the regression functions through step functions and assumes that these step heights come from a Gaussian distribution. Most subsequent theoretical and empirical developments have adopted this specification. In the G-BART setup we assume that the distribution of these step heights belong to a broader family of distributions that include both the Gaussian distribution and also some thicker tailed distributions like Laplace. We demonstrate that the BART model maintains a near-minimax posterior concentration rate, if the step heights come from any of the distributions belonging to this broader family, thus providing a wide range of distributional choices without sacrificing fast posterior concentration.

To support practitioners, we develop \texttt{GBART}, a Python package also accessible in R via \texttt{reticulate}, that implements G-BART for a range of exponential family response variables including Poisson, Inverse Gaussian, and Gamma distributions, alongside the standard continuous regression and binary classification settings. The package is built on \texttt{PyMC-BART} \citep{quiroga2022bayesian}, extending it to handle a broader class of exponential family response distributions beyond the Gaussian case. It uses a Bayesian backfitting algorithm adapted from packages such as \texttt{BART} and \texttt{dbarts}, and provides a user-friendly interface modelled on R's \texttt{glm} function, where users specify a response family and link function, enabling straightforward adoption for users familiar with standard generalized linear model software. We demonstrate the performance of \texttt{GBART} on simulated and real datasets, comparing it against \texttt{BART} and \texttt{dbarts}, and illustrate the effect of link function choice for count regression using simulated data.

This paper is organized as follows. In Section \ref{sec:G-BART} we describe the generalized BART model with the associated priors. Section \ref{sec:Posterior_concentration} discusses the notion of posterior concentration, followed by the main theoretical results on G-BART in Section \ref{sec:results}. Broader implications of these results are described in Section \ref{sec: implications}. The \texttt{GBART} software package along with empirical examples are described in Section \ref{sec:software}. Finally, Section \ref{sec:discussion} concludes with a discussion. Proofs of the main theoretical results are provided in the supplementary material after Section \ref{sec:discussion} .

\subsection{Notations:}
For any two real numbers $a$ and $b$, $a \vee b$ will denote the maximum of $a$ and $b$. The notations $\gtrsim$ and $\lesssim$ will stand for ``greater than or equal to up to a constant'' and ``less than or equal to up to a constant'', respectively. The symbol $P_f$ will abbreviate $\int f dP$ and $\P_{f}^{(n)} = \prod_{i=1}^n \P_{f}^i$ will denote the $n$-fold product measure of the $n$ independent observations, where the $i$-th observation comes from the distribution $P_{f}^i$. Let $H(f,g) = \left (\int (\sqrt{f}-\sqrt{g})^2 d\mu \right )^{1/2}$ and $K(f,g) = \int f \log (f/g) d\mu$ denote the Hellinger distance and the Kullback-Leibler divergence, respectively between any two non-negative densities $f$ and $g$ with respect to a measure $\mu$. We define another discrepancy measure $V(f,g) = \int f \left(\log (f/g)\right)^2 d\mu$. Finally, for any set of real vectors $\X_1, \dots, \X_n \in \mathbb{R}^q$ of size $n$, define the average discrepancy measures $H_n(f,g) = \frac{1}{n}\sum_{i=1}^n H\left(f(\X_i), g(\X_i) \right)$, $K_n(f,g) = \frac{1}{n}\sum_{i=1}^n K\left(f(\X_i), g(\X_i) \right)$ and $V_n(f,g) = \frac{1}{n}\sum_{i=1}^n V\left(f(\X_i), g(\X_i) \right)$, where $f(\theta)$ and $g(\theta)$ denote the densities $f$ and $g$ with respect to parameter $\theta$. Also, for any $L_p$ norm $\norm{\cdot}_p$, define the average norm $\norm{f-g}_{p,n} = \frac{1}{n} \sum_{i=1}^n \norm{f(\X_i)-g(\X_i)}_p$.

\section{The Generalized BART Prior}\label{sec:G-BART}
%The success of ensemble learning methods can be  attributed  to the fact that  many weak learners can perform exceptionally well when deployed collectively. 
The BART method of \cite{chipman2010bart} is a prominent example of Bayesian ensemble learning, where individual shallow trees are entwined together into a forest, that is capable of estimating a wide variety of nonlinear functions with exceptional accuracy, while simultaneously accounting for different orders of interactions among the covariates.
%each learner is a tree, i.e. a step function supported on a tree-shaped partition. 
Building upon BART, we describe a generalized model, where the response variable is assumed to come from an exponential family distribution. For continuous Gaussian response variables, this generalized BART model reduces to the original BART prior of \citep{chipman2010bart}.

The data setup under consideration consists of  $\Y_i = (y_{i1} , \dots , y_{ip} )' \in \mathbb{R}^p$, a set of $p$-dimensional outputs, and $ \X_i = (x_{i1} , \dots , x_{iq} )' \in [0,1]^q $,  a set of $q$ dimensional inputs for $1\leq i\leq n$.  We assume $\Y$ follows some distribution in the exponential family with density of the following form:
    \begin{equation}\label{eq:exponential_pdf}
        P_{f_0}(\Y \C \X)=h(\Y)g\left[f_0(\X)\right]\exp \left [\eta \left (f_0(\X) \right)^T T(\Y) \right],
    \end{equation}
where $h: \mathbb{R}^p \rightarrow \mathbb{R}$, $g: \mathbb{R}^D \rightarrow \mathbb{R}$, $\eta: \mathbb{R}^p \rightarrow \mathbb{R}^J$, $T: \mathbb{R}^p \rightarrow \mathbb{R}^J$ for some integer $J$ and $f_0: \mathbb{R}^q \rightarrow \mathbb{R}^D$, for some integer $D$, are all real valued functions. Among these functions, $h$, $g$, $\eta$ and $T$ are usually {\it known} depending on the nature of the response $\Y$. The function $f_0$, connecting the input $\X$ with the output $\Y$, is the only unknown function and estimating this function is the primary objective of the G-BART estimator. 

We assume that $f_0$ is an unconstrained function, i.e. the range of $f_0$ is the entire space $\mathbb{R}^D$ for some integer $D$. A suitable link function $\Psi(\cdot)$ is used to transform $f_0$ to the natural parameter of the distribution of $\Y$, which is often constrained. For example, for the binary classification problem, $\Y \sim Bernoulli\left (p(\X) \right)$. Here the natural parameter $p(\X) \in (0,1)$ is restricted and hence we can use $\Psi(z) = \frac{1}{1+\exp(-z)}$, the logistic function (or a probit function, as in \cite{chipman2010bart}) to map the unconstrained function $f_0(\X)$ to the natural parameter $p(\X)$. There are usually several different choices for the link function. As we will see in Section \ref{sec: implications}, the BART estimator might have different posterior concentration rates depending on which link function is used to transform the function $f_0$ to the natural parameter of the response distribution.

The univariate regression and the two-class classification problem considered in the original BART paper \citep{chipman2010bart} and many of its important extensions, such as the multi-class classification and the log-linear BART \citep{murray2017log} for categorical and count responses can be formulated as special cases of \eqref{eq:exponential_pdf}. The specific forms of the functions $h, g, \eta$ and $T$ for continuous regression and multi-class classification are given in Table \ref{tab: exponential_table}.

\begin{table}[t]
\caption[Some Examples of Generalized BART Model]{ Univariate Regression (column 2) and Multi-class Classification (column 3), as special cases of the Generalized BART model. $\Phi$ denotes the $Softmax$ function and $\mathcal{M}(\cdot)$ denotes the $Multinomial(1;\cdot)$ distribution. $\left(\{\mathbb{I}\{Y=i\}\}_{i=1}^{p} \right)'$ denotes the row vector where the $i$-th coordinate equals to one if $\Y$ belongs to class $i$ and zero otherwise.  }
\label{tab: exponential_table}
\begin{center}
\begin{small}
\begin{tabular}{cccc}
\toprule
%\hline
         Response ($\Y$)  & Continuous & Categorical \\
\midrule
%\hline \hline
Dist.($\Y$)  &     $\mathcal{N}\left (f_0(\X), \sigma^2 \right )$     & $\mathcal{M}\left(\Phi(f_0(\X))\right)$ \\ 
%\hline
$h(\Y)$   &        $ 1/\sqrt{2\pi \sigma}$   & 1   \\ 
%\hline
$g\left(f_0(\X)\right)$   &  $\exp\left ( -f_0(\X)^2/2\sigma^2 \right)$     & 1  \\ 
%\hline
$\eta\left(f_0(\X) \right)$   &  $\left (f_0(\X), 1 \right )$           & $f_0(\X)$            \\ 
%\hline
$T(\Y)$   &  $\left (2Y/\sigma^2, - Y^2/\sigma^2 \right )$                  & $\left(\{\mathbb{I}\{Y=i\}\}_{i=1}^{p} \right)'$  \\
%\hline
$f_0(\X)$   &  $\mathbb{R}^q \rightarrow \mathbb{R}$               & $\mathbb{R}^q \rightarrow \mathbb{R}^{p-1}$        \\
\bottomrule
%\hline
\end{tabular}
\end{small}
\end{center}
\vskip -0.1in
\end{table}

Next a regression tree is used to reconstruct the unknown function $f_0: \mathbb{R}^q \rightarrow \mathbb{R}^D$  via a mapping $f_{\mT,{\b}}:[0,1]^q \rightarrow \mathbb{R}^D$   so that $f_{\mT,{\b}} ( \X ) \approx f_0 ( \X ) $ for $ \X \notin \{ \X_i \}_{i=1}^n $.
Each such mapping  is essentially a step function of the form
\begin{equation}\label{eq:tree_mapping}
f_{\mT,{\b}}(\X)=\sum_{k=1}^K\beta_k\mathbb{I}(\X\in\Omega_k)
\end{equation}
supported on a tree-shaped partition $\mT=\{\Omega_k\}_{k=1}^K$ and  specified by a vector of step heights $\b=(\beta_1,\dots,\beta_K)'$.
The vector $\beta_k \in \mathbb{R}^p$ represents the value of the expected response inside the $k$-th cell of the partition $\Omega_k$.

Bayesian additive trees consist of an ensemble of multiple shallow trees, each of which is intended to be a weak learner, geared towards addressing a slightly different aspect of the prediction problem. These trees are then woven into an {\it additive} forest mapping of the form
\begin{equation}\label{eq:forest_mapping}
f_{\mE,{\B}}(\x)=\sum_{t=1}^Tf_{\mT_t,{\b}_t}(\x),
\end{equation}
where each $f_{\mT_t,{\b}_t}(\x)$ is of the form \eqref{eq:tree_mapping}, $\mE=\{\mT_1,\dots,\mT_T\}$ is an ensemble of $T$ trees and $\B=\{\b_1,\dots,\b_T\}'$ is a collection of jump sizes corresponding to the $T$ trees. 

Since each individual member of the approximating space is a step function of the form \eqref{eq:forest_mapping}, supported on a Bayesian additive forest, the prior distribution should include three components: (i) a prior $\pi(T)$ on the number of trees $T$ in the ensemble, (ii) a prior on individual tree partitions $\pi(\mT)$ and their collaboration within the ensemble and (iii) given a single tree partition $\mT$, a prior $\pi(\b\C\mT)$ has to be imposed on the individual step heights $\b$. 

In this paper we follow the recommendation by \cite{chipman2010bart} and assume the number of trees $T$ to be fixed at a large value (e.g. $T=200$ for regression and $T=50$ for classification). This is equivalent to assigning a degenerate prior distribution on $T$, where all probability mass is concentrated on a single positive integer. Alternatively, one can also assign a prior with higher dispersion, as in \cite{rockova2020posterior} and \cite{linero2017bayesian} and replicate the steps of the proofs provided in the appendix with minor modifications. 

Given the total number of trees in the ensemble, individual trees are assumed to be independent and identically distributed with some distribution $\pi(\mT)$. This reduces the prior on the ensemble to be of the form
\begin{equation}\label{eq:joint_prior}
\pi(\mE,\B)=\prod_{t=1}^T\pi(\mT_t)\pi(\b_t\C\mT_t),
\end{equation}
where $\pi(\mT_t)$ is the prior probability of a partition $\mT_t$, while $\pi(\b_t\C\mT_t)$ is the prior distribution over the jump sizes.
The specific forms of the priors $\pi(\mT)$ and $\pi(\b \C \mT)$ are described below.

\subsection{\bf Prior on partitions}\label{sub:prior_tree} 

We consider two distinct prior distributions on the following class of step functions
\begin{align}
\mF = \{f_{\mE,{\B}} \text{ of the form \eqref{eq:forest_mapping} for some $\mE$ and $\B$}\}.
\end{align}
Specifically, the prior on individual tree partitions $\pi(\mT)$ follow the specifications proposed by \cite{chipman1998bayesian} or by \cite{denison1998bayesian}, respectively. The posterior concentration results discussed in Section \ref{sec:results} are applicable to both these priors. \cite{chipman1998bayesian} specifies the prior over trees implicitly as a tree generating stochastic process, described as follows:
\begin{enumerate}
\item Start with a single leaf (a root node) encompassing the entire covariate space.
\item Split a terminal node, say $\Omega$, with a probability
\begin{equation}\label{eq:p_split}
p_{split}(\Omega) \propto \alpha^{-d(\Omega)} \text{ for some $0 < \alpha < 1/2$.}
\end{equation}
where $d(\Omega)$ is the depth of the node $\Omega$ in the tree architecture. This choice, motivated by \citep{rockova2019theory}, is slightly different from the original prior of \cite{chipman1998bayesian}, who assumed $p_{split}(\Omega) \propto \alpha/(1+d(\Omega))^{-\beta}$ for $\beta \ge 0$. \footnote{The reason behind this modification is that the original BART prior of \cite{chipman2010bart} does not decay at a fast enough rate. However since we examine only sufficient (but not necessary) conditions for optimal posterior concentration, our results do not guarantee that the original prior is inherently worse than the modified prior. The main obstacle to extending the current proof to the original prior is that its splitting probability does not decay fast enough to satisfy condition (C3). Addressing this would likely require a more refined analysis of the prior tail behavior, which we leave for future work.}
\item If the node $\Omega$ splits, assign a splitting rule  and create left and right children nodes. The splitting rule consists of picking a split variable $j$ uniformly from available directions $\{1,\dots, p\}$ and
picking a split point $c$ uniformly at random from the data values $\{x_{ij} : \X_i \in \Omega\}$ falling within the current cell $\Omega$.
\end{enumerate}

A description of the prior proposed by \cite{denison1998bayesian} is given in Section \ref{denison_prior} in the supplementary material.
\subsection{\bf Prior on step heights}\label{sub:prior_step} 
We impose a broad class of priors on the step heights that incorporate the corresponding component of the classical BART model as a special case. Given a tree partition $\mT_t$ with $K_t$ steps, \cite{chipman2010bart} considers identically distributed independent Gaussian jumps with mean $0$ and variance $\sigma^2$. In the G-BART set-up we assume that the $j$-th step height of the $t$-th tree, $\beta_{tj} \stackrel{i.i.d}{\sim} F_\beta$, where $F_\beta$ is any general distribution with the following property: for some constants $C_1, C_2, C_3$ such that $C_1>0$, $0 < C_2 \le 2$ and $C_3 > 0$,
\begin{equation}\label{eq:beta_tail1}
    F_\beta(\norm{\beta}_\infty \le t) \gtrsim \left (e^{-C_1 t^{C_2}} t \right)^{p} \quad \text{for $0< t \le 1$}
\end{equation}
and
\begin{equation}\label{eq:beta_tail2}
    F_\beta(\norm{\beta}_\infty \ge t) \lesssim e^{-C_3 t} \quad \text{for $t \ge 1$}
\end{equation}
where $\norm{\cdot}_\infty$ represents the $L_\infty$ norm and $F_\beta(\norm{\beta}_\infty \ge t)$ denotes the tail probability of the distribution on the step heights $\beta \in \mathbb{R}^p$. Both the multivariate Gaussian and the multivariate Laplace distribution come from this family of distributions and so do any sub-Gaussian distributions. A proof of these statements is provided in the appendix. We will see in Section \ref{sub: step_result} - \ref{sub: holder_result} that these conditions are {\it sufficient} to guarantee that the G-BART estimator has a near-optimal posterior concentration rate.

However we should note that the conditions \eqref{eq:beta_tail1}-\eqref{eq:beta_tail2}, although {\it sufficient}, are not {\it necessary} conditions and distributional assumptions on the step sizes that do not satisfy these conditions, might still guarantee a near-optimal posterior concentration rate. For such an example, please refer to the `classification with Dirichlet step-heights' in Section \ref{sec: implications}.

\section{Posterior Concentration}\label{sec:Posterior_concentration}
Posterior concentration statements are a prominent artifact in Bayesian nonparametrics, where the primary motivation is to examine the quality of a Bayesian procedure, by studying the learning rate of its posterior, i.e. the rate at which the posterior distribution, centralizes around the truth as the sample size $n\rightarrow\infty$. In empirical settings, posterior concentration results have often influenced the proposal and fine-tuning of priors. Oftentimes seemingly unremarkable priors give rise to capricious outcomes, especially in the infinite-dimensional parameter spaces, such as the one considered here (\cite{cox1993analysis}, \cite{diaconis1986consistency}) and designing well-behaved priors turns out to be of utmost importance, thus further reinstating the importance of posterior concentration statements.

The Bayesian approach proceeds by imposing a prior measure $\Pi(\cdot)$ on $\mF$, the set of all estimators of $f_0$. For the G-BART models this corresponds to the set of all step functions supported on an additive ensemble of Bayesian trees. Given observed data $\Y^{(n)}=(Y_1,\dots, Y_n)'$, the inference about $f_0$ is solely dependent on the posterior distribution
$$
\Pi(A\mid\Y^{(n)})=\frac{\int_A \prod_{i=1}^n \Pi_f(Y_i\mid\X_i)\d\Pi(f)}{\int \prod_{i=1}^n \Pi_f(Y_i\mid\X_i)\d\Pi(f)}\quad\forall A\in\mathcal{B}
$$
where $\mathcal{B}$ is a $\sigma$-field on $\mF$ and 
where $\Pi_f(Y_i\C\X_i)$ is the conditional likelihood function for the output $Y_i$, given the covariates $\X_i$, under parameterization $f$.

Ideally under a suitable prior, the posterior should put most of its probability mass around a small neighborhood of the true function and as the sample size increases, the diameter of this neighborhood should go to zero at a fast pace. Formally speaking, for a given sample size $n$, if we examine an $\varepsilon_n$-neighborhood of the true function $\mathcal{A}_{\varepsilon_n}$, for some $\varepsilon_n \rightarrow 0$ and $n \varepsilon_n^2 \rightarrow \infty$, we should expect
\begin{equation}\label{eq:exponential_concentration}
\Pi(\mathcal{A}_{\varepsilon_n}^c \C \Y^{(n)} ) \rightarrow 0\quad\text{in $\P_{f_0}^{(n)}$-probability as $n\rightarrow\infty$,} 
\end{equation}
where $\mathcal{A}_{\varepsilon_n}^c$ denotes the complement of the neighborhood $\mathcal{A}_{\varepsilon_n}$.

In the context of G-BART, given observed data $\Y^{(n)}=(\Y_1,\dots, \Y_n)'$, we are interested in evaluating whether the posterior concentrates around the true likelihood $\P_{f_0}^{(n)} = \prod_{i=1}^n P_{f_0}^{i}$ at a near-minimax rate, where $P_{f_0}^i = P_{f_0}(\Y_i \C \X_i)$ is of the form \eqref{eq:exponential_pdf}, for $i=1,\dots,n$. Following the suggestions of \cite{ghosal2007convergence}, we look at the smallest $H_n$-neighborhoods around $\P_{f_0}^{(n)}$ that contain the bulk of the posterior probability. Specifically, for a diameter $\varepsilon>0$ define 
\begin{equation}\label{eq:A_neighborhood}
\mathcal{A}_{\varepsilon}=\{f \in \mF: H_n(P_f,P_{f_0}) \leq  \varepsilon \}
\end{equation}

Theorem 4 of \cite{ghosal2007convergence} demonstrates that the statement \eqref{eq:exponential_concentration} can be proved by verifying three sufficient conditions. The first condition, henceforth referred to as the ``entropy condition'' specifies that
\begin{equation}\label{eq:entropy1}
\displaystyle \sup_{\varepsilon > \varepsilon_n} \log  N\left(\tfrac{\varepsilon}{36}; \mathbb{F}_n \cap \mathcal{A}_{\varepsilon}; H_n\right) \lesssim n\,\varepsilon_n^2, \tag{C1}
\end{equation}
where $N(\varepsilon; \Omega; d)$ denotes the $\varepsilon$-covering number of a set $\Omega$ for a semimetric $d$,  i.e. the minimal number of $d$-balls of radius $\varepsilon$ needed to cover the set $\Omega$ and $\{\mathbb{F}_n\}_{n \ge 1}$ denotes an increasing sequence of approximating sieves. The sequence of sieves used in this paper is described in the appendix.

The second condition requires that the prior puts enough mass around the true likelihood $\P_{f_0}^{(n)}$, meaning that for a given sample size $n \in \mathbb{N}\setminus \{0\}$ and for some $d>2$,
\begin{equation}\label{eq:prior1}
\displaystyle {\Pi(f \in \mathcal{F}: K_n(f_0,f) \vee V_n(f_0,f) \leq \varepsilon_n^2 )} \gtrsim e^{-d\,n\,\varepsilon_n^2}, \tag{C2}
\end{equation}
where $K_n(f_0,f) = \frac{1}{n}\sum_{i=1}^n K(P_{f_0}(\cdot\mid\X_i), P_f(\cdot\mid\X_i))$ and $V_n(f_0,f) = \frac{1}{n}\sum_{i=1}^n V(P_{f_0}(\cdot\mid\X_i), P_f(\cdot\mid\X_i))$ are the averaged Kullback-Leibler divergence and second-order KL variation from the truth.

The final condition, referred to as the ``prior decay rate condition'' stipulates that the sequence of sieves $\mathbb{F}_n \uparrow \mF$ captures the entire parameter space with increasing accuracy, in the sense that the complementary space $\mathcal{F} \backslash \mathbb{F}_n$ has negligible prior probability mass for large values of $n$.
\begin{equation}\label{eq:remain1}
\displaystyle \Pi(\mathcal{F} \backslash \mathbb{F}_n) = o(e^{-(d+2)\,n\,\varepsilon_n^2})  \tag{C3}
\end{equation}
The results of type \eqref{eq:exponential_concentration} quantify not only the typical distance between a point estimator (posterior mean/median) and the truth, but also the typical spread of the posterior around the truth and hence are stronger than `posterior consistency' statements. These results are usually the first step towards further uncertainty quantification statements, e.g., Bernstein-von Mises theorem \citep{castillo2014bernstein}.

\section{Main Results}\label{sec:results}
In this section we describe our main theoretical findings, which describe the posterior concentration rates of the generalized Bayesian trees and their additive ensembles (G-BART), when the true function $f_0$ connecting the response $\Y$ with the covariates $\X$, is either (a) a step function (Theorem \ref{theorem: step}), or (b) a monotone function (Theorem \ref{theorem: monotone}), or (c) a $\nu$-H\"{o}lder continuous function with $0 < \nu \le 1$ (Theorem \ref{theorem: smooth}). We make two important assumptions: the first assumption (subsequently referred to as Assumption 1), given below restricts the distribution of the response variable $\{\Y_1,\dots,\Y_n\} \in \mathbb{R}^{p}$ to a specific class of exponential family distributions while the second assumption (subsequently referred to as Assumption 2) concerns the spread of the covariates $\{\X_1, \dots, \X_n\} \in \mathbb{R}^{q}$.

\paragraph{Assumption 1:}
Let $\Y_1, \dots, \Y_n  \sim P_{f}$, where $P_{f}$ denotes a probability density function of the form \eqref{eq:exponential_pdf}, such that, $\eta(z) = z$ and there exist strictly increasing positive sequences $\{C^n_{g}\}_{n \ge 1}$ and $\{C^n_\beta\}_{n \ge 1}$, such that
\begin{equation}\label{eq:g_condition}
    \abs{\frac{\nabla g (\b)}{g(\b)}} \le C^n_{g} \textbf{1}_p, \quad \forall \b \in B_n= \left \{\b: \norm{\b}_\infty \le C^n_\beta \right\},
\end{equation}
where $\textbf{1}_p = (1,\dots,1) \in \mathbb{R}^p$ denotes a $p$-dimensional vector of ones and $\nabla g$ denotes the vector of partial derivatives. We assume $\{C^n_{g}\} \vee \{C^n_\beta\} \lesssim n^M$ for some $M > 0$. 

\iffalse
The significance is that the function $g(\cdot)$ should not change too rapidly, and the higher the sample size the larger the rate of change is allowed. The above assumption is satisfied by most distributions commonly used in the regression and classification settings, as will be demonstrated in Section \ref{sec: implications}.
\fi

Intuitively, Assumption 1 ensures that the likelihood is not overly sensitive to perturbations in $f_0$. If $g(\cdot)$ were to change rapidly, small differences between two candidate functions could produce very different likelihood values, making it difficult to control the KL divergence between them. By bounding $|\nabla g(\beta)/g(\beta)|$, Assumption 1 guarantees that closeness in function space implies closeness in likelihood space, a property that is essential for verifying condition \eqref{eq:prior1}.

The above assumption is satisfied by most distributions commonly used in the regression and classification settings, as will be demonstrated in Section \ref{sec: implications}.

\paragraph{Assumption 2:} For a k-d tree partition, $\wh{\mathcal{T}} = \{\wh{\Omega_k}\}$, with $K=2^{p s}$-many leaves\footnote{Here $s$ is a non-negative integer indexing the resolution of the k-d tree partition, determining the number of splits applied to each coordinate direction. It depends on $n$: $s$ is chosen so that $K$ grows with $n$ at a rate 
that balances approximation error and prior concentration, as described in 
Supplementary Material \ref{sec:result_proof}.},  the dataset $\{\X_1, \dots, \X_n\}$ satisfies the following condition: for any nonnegative integer $s$, there exists some large enough constant $M > 0$ such that
\begin{equation}\label{eq:regular_dataset}
\max_{1 \le k \le K} \text{diam}(\wh{\Omega_k}) < M \sum_{k=1}^K \mu(\Omega_k) \text{diam}(\wh{\Omega_k}),
\end{equation}
where $\mu(\Omega_k) = \frac{1}{n} \sum_{i=1}^n \mathbb{I}\{\X_i \in \Omega_k\}$ denotes the proportion of observations in the cell $\Omega_k$ and $\text{diam}(\wh{\Omega_k}) = \max_{\x,\y \in \Omega_k} \norm{\x - \y}_2$ denotes the spread of the cell $\Omega_k$ with respect to the $L_2$-norm.

\subsection{Results on Step-Functions}\label{sub: step_result}
Let us suppose $f_0$ is a step function supported on an axes-paralleled partition $\{\Omega_k\}_{k=1}^{K_0}$. For any such step function $f_0$, we define the {\it complexity of $f_0$}, as the smallest $K$ such that there exists a k-d tree partition $\{\Omega_k\}_{k=1}^{K}$ with $K$ cells, for which the step function $f(x) = \sum_{k=1}^K \beta_k \mathbb{I}\{x \in \Omega_k\}$ can approximate $f_0$ without any error, for some step heights $(\beta_1, \dots, \beta_K) \in \mathbb{R}^K$. This complexity number, denoted by $K_{f_0}$, depends on the true number of step $K_0$, the diameter of the intervals $\{\Omega_k\}_{k=1}^{K_0}$, and the number of covariates $q$. The actual minimax rate for approximating such piecewise-constant functions $f_0$ with $K_0 > 2$ pieces, is $n^{-1/2}\sqrt{K_0 \log {(n/K_0)}}$ \citep{gao2017minimax}. The following theorem shows that the posterior concentration rate of G-BART is almost equal to the minimax rate, except that $K_0$ gets replaced by $K_{f_0}$. The discrepancy is an unavoidable consequence of the fact that the true number of steps $K_0$ is unknown. Had this information been available, the G-BART estimator would have attained the exact minimax rate.

\begin{theorem}\label{theorem: step}
If we assume that the distribution of the step-sizes satisfies \eqref{eq:beta_tail1} and \eqref{eq:beta_tail2}, then under Assumptions 1 and 2 with $q \lesssim \sqrt{\log n}$, the generalized BART estimator satisfies the following property:

If $f_0$ is a step-function, supported on an axes-paralleled partition, with complexity $K_{f_0} \lesssim \sqrt{n}$ and $\norm{f_0}_\infty \lesssim \sqrt{\log n}$, then with $\varepsilon_n=n^{-1/2}\sqrt{K_{f_0} \log^{2\gamma}{\left(n/K_{f_0}\right)}}$ and $\gamma > 1/2$,
\begin{equation*}
\Pi\left( f \in \mathcal{F}: H_n(\P_f,\P_{f_0}) > M_n \varepsilon_n \mid \Y^{(n)}\right) \to 0,
\end{equation*}
for any $M_n \rightarrow \infty$ in $\P_{f_0}^{(n)}$-probability, as $n,q \to \infty$. 
\end{theorem}
\vspace{-10pt}
\begin{proof}
Proof is given in the appendix.
\end{proof}

\subsection{Results on Monotone Functions}\label{sub: monotone_result}
An important implication of Theorem \ref{theorem: step} is that posterior concentration results on step functions can potentially build the foundation for similar results on broader class of functions, aided by the ``simple function approximation theorem'' \citep{stein2009real}, which states that for any measurable function $f$ on $\mathcal{E} \subseteq \mathbb{R}^q$, there exists a sequence of step functions $\{f_k\}$ which converges point-wise to $f$ almost everywhere \citep{stein2009real}. As a corollary to this theorem, we can derive the following result on the set of all coordinatewise monotone functions. A function $f_0: \mathbb{R}^q \rightarrow \mathbb{R}$ is defined as coordinatewise monotone if $f_0(\textbf{x})$ is monotone (increasing or decreasing) in $x_j$ for every $j \in \{1, \dots, q\}$, with all other coordinates $\{x_1, \dots, x_q\} \setminus \{x_j\}$ held fixed.

\begin{lemma}\label{lemma:monotone_approx}
Any coordinatewise {\bf monotone} bounded function $f_0$ can be approximated with arbitrary precision $\varepsilon$, by a step function supported on a k-d tree partition with number of leaves $K_{f_0}(\varepsilon) \ge \lceil 1/\varepsilon \rceil$. We define $K_{f_0}(\varepsilon)$ to be the {\it complexity of the monotone function $f_0$} with respect to $\varepsilon > 0$.
\end{lemma}
The complexity $K_{f_0}(\varepsilon)$ also depends on the dimension of the domain $q$ as well as on the magnitude of the true function $\norm{f_0}_\infty$. This paves the way for deriving the posterior concentration rate of G-BART when the true function $f_0(\cdot)$ connecting the covariates $\X$ with a univariate response $\Y$ is a monotone function. The minimax rate of estimation for such densities is $n^{-1/(2+q)}$ \citep{biau2003risk}. The following theorem states that the posterior concentration rate of G-BART equals to this optimum rate up to a logarithmic function, provided that the magnitude of the true function $f_0$ is not ``too large''.

\begin{theorem}\label{theorem: monotone}
If we assume that the distribution of the step-sizes satisfies \eqref{eq:beta_tail1} and \eqref{eq:beta_tail2}, then under Assumptions 1 and 2 with $q \lesssim \sqrt{\log n}$, the generalized BART estimator satisfies the following property:

If the true function $f_0: \mathbb{R}^q \rightarrow \mathbb{R}$ is coordinatewise monotone, with $\norm{f_0}_\infty \lesssim \sqrt{\log n} $, then with $\varepsilon_n = n^{-1/(2+q)} \sqrt{\log n}$,
\begin{equation*}
\Pi\left( f \in \mathcal{F}: H_n(\P_f,\P_{f_0}) > M_n \varepsilon_n \mid \Y^{(n)}\right) \to 0,
\end{equation*}
for any $M_n \rightarrow \infty$ in $\P_{f_0}^{(n)}$-probability, as $n,q \to \infty$. 
\end{theorem}
%\vspace{-10pt}
\begin{proof}
The first step of the proof involves finding an approximating step-function $\wh{f_0}$ by Lemma \ref{lemma:monotone_approx}, such that $\norm{f_0-\wh{f}_0}_{2,n} < \varepsilon_n/2$. The rest of the proof follows by retracing the steps as in the proof of Theorem \ref{theorem: smooth} given in the supplementary material.
\end{proof}

The above result demonstrates that the Generalized BART model adapts to monotonic patterns in the true function $f_0$, without any additional prior assumptions.

\subsection{Results on H\"{o}lder Continuous Functions}\label{sub: holder_result}
This section describes the posterior concentration results on G-BART when the true function $f_0$ connecting $\X$ with $\Y$ is a $\nu$-H\"{o}lder continuous function with $0< \nu \le 1$. \cite{rockova2020posterior} and \cite{rockova2019theory} proved that the posterior concentration rates of the BART model (under the priors of \cite{denison1998bayesian} and \cite{chipman2010bart} respectively) are equal to $n^{-\alpha/(2\alpha+q)}$, the minimax rate of estimation for such functions \citep{stone1982optimal}, except for a logarithmic factor. These results can be derived as direct corollaries of the following theorem for G-BART, when $\Y$ is a univariate continuous response and the step-sizes are assumed to follow a Gaussian distribution.

\begin{theorem}\label{theorem: smooth}
If we assume that the distribution of the step-sizes satisfies \eqref{eq:beta_tail1} and \eqref{eq:beta_tail2}, then under Assumptions 1 and 2 with $q \lesssim \sqrt{\log n}$, the generalized BART estimator satisfies the following property: 
 
If $f_0$ is a $\nu$-H\"{o}lder continuous function with  {$0<\nu\leq 1$}, where $\|f_0\|_\infty\lesssim \sqrt{\log n}$, then with $\varepsilon_n=n^{-\alpha/(2\alpha+q)}\sqrt{\log n}$,
\begin{equation*}
\Pi\left( f \in \mathcal{F}: H_n(\P_f,\P_{f_0}) > M_n \varepsilon_n \mid \Y^{(n)}\right) \to 0,
\end{equation*}
for any $M_n \rightarrow \infty$ in $\P_{f_0}^{(n)}$-probability, as $n,q \to \infty$. 
\end{theorem}
\vspace{-10pt}
\begin{proof}
Proof is given in the appendix.
\end{proof}

\paragraph{Remark:}
Interestingly, the posterior concentration rates derived in Theorems \ref{theorem: step}-\ref{theorem: smooth}, do not depend on the number of trees $T$ in the generalized BART ensemble. In other words the concentration rate is equally valid for a single tree (i.e. $T=1$), as well as for tree ensembles (i.e. $T > 1$), when the true regression function $f_0$ is $\nu$-H\"older continuous with $0 < \nu \le 1$. However as has been seen in multiple empirical applications \citep{chipman2010bart}, Bayesian forests consisting of multiple trees provide superior out-of-sample predictive performance, compared to a single tree, the reason being that multiple weak tree learners, when woven together into a forest, can accommodate a wider class of partitions, as opposed to a single tree. 

This phenomenon can be reinforced by theoretical results, such as Theorem 6.1 of \cite{rockova2020posterior}. When the true function $f_0$ is of the form $f_0 = \sum_{t=1}^{T_0} f_0^t$, where $f_0^t$ is a $\nu_t$-H\"older continuous function, with $0 \le \nu^t \le 1$ and $T_0 \to \infty$, a forest with multiple trees have a posterior concentration rate equal to $\varepsilon_n^2 = \sum_{t=1}^{T_0} n^{-2\nu_t/(2\nu_t+p)}\log n$, provided $T_0 \lesssim n$, whereas single regression trees fail to recognize the additive nature of the true function and attain a slower concentration rate. A similar result is presented in Theorem 4 of \cite{linero2017bayesian}, under a kernel-smoothed version of the BART prior. 

Although the BART prior considered by \cite{rockova2020posterior} is fundamentally different from the classical BART prior \citep{chipman2010bart} considered here, their result on additive functions can be replicated in the present set up as well, provided we allow the number of trees $T$ in the BART ensemble to be stochastic. In particular, we might assume that $\pi(T) \propto e^{-C_T T}$, for $T \in \mathbb{N} \setminus \{0\}$, with $C_T > \log 2$,
thus enabling the number of trees in the forest to adapt to unknown $T_0$, as $n, p \rightarrow \infty$.

\section{Implications}\label{sec: implications}

The primary significance of Theorems \ref{theorem: step}, \ref{theorem: monotone} and \ref{theorem: smooth} is that these results provide a frequentist theoretical justification for superior empirical performance of generalized Bayesian trees and forests, claiming that the posterior concentrates around the truth at a near-optimal learning rate. As demonstrated below, we can show that the original BART model \citep{chipman2010bart}, along with some of its commonly used variants (such as BART for multi-class classification and regression on count data) have near-optimal posterior concentration rates, as direct corollaries of Theorems \ref{theorem: step} - \ref{theorem: smooth}. Another important consequence of these results is that (see Section \ref{parsimony} of the supplementary material), they show that the posterior distribution on the number of leaves in a generalized Bayesian tree does not exceed the optimal number of splits by more than a constant multiple and hence are resilient to overfitting.

Below we demonstrate the breadth of applicability of Theorems \ref{theorem: step}, \ref{theorem: monotone} and \ref{theorem: smooth} in proving analogous theoretical results for a wide range of commonly used BART models. 

\paragraph{Continuous Regression:} 
For a (multivariate) continuous regression, assume that the response $\Y \C \X \sim \mathcal{N}_p(\bm{\mu}(\X),\Sigma)$, for some positive definite $\Sigma$. The function $g(f_0(\X)) = g(\bm{\mu})=e^{-\bm{\mu}^T \Sigma^{-1}\bm{\mu}/2}$ satisfies \eqref{eq:g_condition} with $B_n=[-n,n]^p$ and $C_g^n= n \lambda(\Sigma) $, where $\lambda(\Sigma)$ denotes the maximum eigenvalue of $\Sigma$. Hence from Theorems \ref{theorem: step}, \ref{theorem: monotone} and \ref{theorem: smooth}, we can conclude that for continuous regression, the G-BART estimator has a near-minimax posterior concentration rate, provided that the true function $f_0$ connecting the input $\X$ with the output $\Y$ is either a step function, a monotone function or a $\nu$-H\"{o}lder continuous function with $0 < \nu \le 1$.

\paragraph{Classification with Gaussian Step Heights:}
For a $p$-class classification the response $\Y$ can be written as a $p$ dimensional binary vector that has $1$ at the $l$-th coordinate if $\Y$ belongs to category $l \in \{1,\dots,p\}$ and $0$ elsewhere. We can assume $\Y \C \X \sim \text{Multinomial}(1; \bm{\pi}(\X))$
for some $\bm{\pi}: \mathbb{R}^q \in (0,1)^p$ such that $\bm{\pi}'\textbf{1}_p = 1$. The unrestricted function $f_0(\X)$ can be transformed to the natural parameter $\pi(\X)$ by a logistic (softmax) or an inverse-probit link function \citep{chipman2010bart} denoted by $\Psi(\cdot)$, so that $\pi(\X) = \Psi(f_0(\X))$. In either case, the function $g(f_0(\X)) = 1$ trivially satisfies condition \eqref{eq:g_condition}. Hence from Theorem \ref{theorem: step} and Theorem \ref{theorem: smooth}, we can conclude that the BART model for multi-class classification has a near-minimax posterior concentration rate.

\paragraph{Classification with Dirichlet Step-Heights}
For the same multi-class classification problem with $p$ classes described above, an alternative prior specification is recommended by \cite{denison1998bayesian}. The parameters $\bm{\pi}(\X)$ can be approximated by multivariate step functions of the form $f_{\mT,{P}}(\x)=\sum_{k=1}^K P_k\mathbb{I}(\x\in\Omega_k)$
on a tree-partition $\{\Omega_k\}_{k=1}^K$. \cite{denison1998bayesian} assumes that $P_k = \left (P_{k1},\dots,P_{kp} \right ) \stackrel{i.i.d}{\sim} \text{Dirichlet}(\alpha_1,\dots,\alpha_p)$, where $\alpha_l > 0, \quad \forall l \in \{1,\dots,p\}$. For example, in a binary classification ($p=2$) problem, we can assign prior $P_k \stackrel{i.i.d}{\sim} \text{Beta}(2,2)$ on the step-heights. The prior $\text{Beta}(2,2)$ violates condition \eqref{eq:beta_tail1}. But we can show that this estimator has a near-optimal posterior concentration rate, even if we cannot conclude this from the results discussed in Section \ref{sec:results}. A proof is given in the supplementary material. This demonstrates that the assumptions we make in Section \ref{sec:results} are merely {\it sufficient} but not {\it necessary} conditions for proving that the generalized Bayesian tree estimator has a near-minimax posterior concentration rate.

\paragraph{Count Regression:}
For count response variable, $\Y \sim Poisson \left[\lambda(\X) \right]$ with $\lambda(\X) > 0$. There are several choices for the link function $\Psi(\cdot)$ to map the unconstrained function $f_0(\X)$ to the constrained parameter $\lambda(\X)$. The posterior concentration rate of the Generalized Bayesian tree estimator might differ depending on which link function is used. For example, if we use $\Psi(z) = \log \left( 1 + \exp(z)\right)$, the softplus link function, then $g(f_0(\X))= 1/(1 + \exp \left (f_0(\X) \right ) $, trivially satisfies condition \eqref{eq:g_condition} and we can conclude that the generalized tree estimator has a near-minimax concentration rate from Theorems \ref{theorem: step}, \ref{theorem: monotone} and \ref{theorem: smooth}.

In contrast, if we use $\Psi(z) = \exp(z)$ as the link function, then $g(f_0(\X))= \exp \left(-\exp(f_0(\X))\right)$ does not satisfy the condition \eqref{eq:g_condition}, when the true function $f_0$ is a $\nu$-H\"{o}lder continuous function. Therefore we cannot apply Theorem \ref{theorem: smooth} anymore to imply that the generalized tree estimator has a near-optimal rate of posterior concentration. When $f_0$ is a step function with complexity $K_{f_0}$, the condition \eqref{eq:g_condition} is satisfied with $B_n=[-K_{f_0}\log n,K_{f_0}\log n]$ and $C_g^n=n^{K_{f_0}}$. The posterior concentration rate becomes $\varepsilon_n=n^{-\frac{1-\alpha}{2}}\sqrt{K_{f_0} \log^{2\eta}(n/ K_{f_0})}$ under the assumption $K_{f_0} \lesssim n^{\alpha}$ for some $0< \alpha < 1$. This is slower than the near-optimal concentration rate $n^{-\frac{1}{2}}\sqrt{K_{f_0}\log^{2\eta}(n/K_{f_0})}$, if we use $\Psi(z) = \log \left( 1 + \exp(z)\right)$, the softplus link function, instead.

Despite these theoretical limitations, our empirical experiments on simulated data (Section \ref{sec:software}, Figures \ref{fig:deviance_links}) show that the canonical `log' link performs comparably to the `softplus' link with deviance being only slightly higher than those for `softplus' link for most examples. The difference in performance between the two link functions was not statistically significant (Wilcoxon rank-sum test $p=0.587$). This reinforces the observation that the assumptions in Section \ref{sec:results} are sufficient but not necessary conditions for near-minimax concentration. While we could not provide a formal proof for the `log' link, the empirical results suggest the log link may still achieve a near-minimax rate, albeit there is a possibility of a slightly slower multiplier in the concentration rate, than the multiplier $\sqrt{\log {n}}$ in Theorems \ref{theorem: monotone}-\ref{theorem: smooth}.

\section{GBART Software}\label{sec:software}
We provide software for implementing BART models for exponential family response variables in both R and Python. The \texttt{GBART} Python package enables model fitting through a single function call, supporting all exponential family distributions and link functions available in the \texttt{glm} function in R. The package is built as a wrapper around the PyMC-BART library, extending it with a streamlined interface for generalized BART modeling that mirrors the syntax familiar to statisticians working with generalized linear models. To facilitate adoption across disciplines, we provide a quick start guide for R and Python users, alongside the package at \url{https://github.com/Enakshi-Saha/gbart}. The package contains unit tests for reliability and verification of any new future updates.

The \texttt{GBART} package supports all exponential family distributions and their canonical link functions, as well as additional link functions identified by the posterior concentration theory developed in this paper, specifically, alternative link functions that carry theoretical guarantees of near-minimax posterior concentration rates under the conditions established in Theorems \ref{theorem: step}-\ref{theorem: smooth} (e.g. the `softplus' link for the Poisson family). Table \ref{tab:gbart_links} summarizes the full set of supported distributions and link functions. For R users, the package is accessible via the \texttt{reticulate} interface, with detailed instructions provided in the GitHub repository.

\paragraph{Remark:} Although the theoretical results in this paper are established for a modified tree prior with faster-decaying splitting probabilities, the \texttt{GBART} package implements the original BART prior of Chipman et al. (2010), which is the conventional choice across all existing BART software and has shown strong empirical performance in practice.

We compare \texttt{GBART} with the two most widely used R packages for BART regression and classification, \texttt{BART} and \texttt{dbarts}, on simulated and real datasets, to ensure that \texttt{GBART} maintains predictive performance of these models, compared to existing software, while providing an easy interface for Python users, along with additional functionalities for other exponential family distributions. Across these benchmarks, \texttt{GBART} achieves comparable predictive performance to both packages. Beyond replicating existing functionality, \texttt{GBART} offers two distinct advantages: it brings BART modeling to researchers working in Python, and it extends the BART framework to the full exponential family, supporting response distributions not available in any existing software.

Regarding computational cost, \texttt{GBART} generates 10,000 MCMC samples after 5,000 burn-in iterations with a median training time of approximately 19 minutes for regression and 27 minutes for classification on a node equipped with dual Intel Xeon Platinum 8358 CPUs (64 cores, 2.60GHz) and 251 GB RAM, which is practical for most applied settings. The longer runtime relative to \texttt{BART} and \texttt{dbarts} reflects the difference between Python and the optimized C++ backends underlying the R packages, rather than any algorithmic inefficiency. For practitioners working primarily with logit or probit models for classification and regression with Gaussian responses in R, we recommend using the existing R packages as the most efficient software choice. For Python users and for non-Gaussian exponential family responses, \texttt{GBART} fills a critical gap by offering a useful complement to existing software.

We note that \texttt{BART} provides specialized models such as survival analysis that are outside the current scope of \texttt{GBART}, and we view these as complementary tools rather than direct competitors. Table~\ref{tab:compute_time} reports computation times across packages on simulated data.
\begin{table}[ht]
\centering
\caption{Exponential family distributions and supported link functions 
         in the \texttt{GBART} package.}
\label{tab:gbart_links}
\begin{tabular}{ll}
\toprule
\textbf{Family} & \textbf{Supported Link Functions} \\
\midrule
Gaussian         & identity, log, inverse \\
Binomial         & logit, probit, cauchit, log, cloglog \\
Gamma            & inverse, identity, log, softplus \\
Poisson          & log, softplus, identity, sqrt \\
Inverse Gaussian & inverse squared, inverse, identity, log \\
\bottomrule
\end{tabular}
\end{table}

% Figure 1: Simulation boxplots side by side
\begin{figure}[htbp]
    \centering
    \begin{minipage}{0.48\textwidth}
        \centering
        \includegraphics[width=\textwidth]{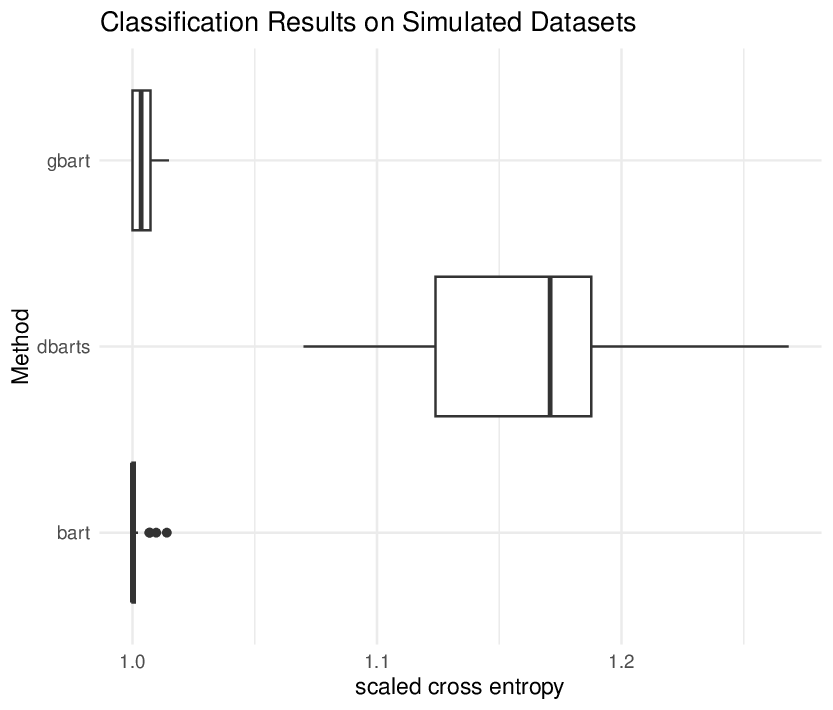}
    \end{minipage}
    \hfill
    \begin{minipage}{0.48\textwidth}
        \centering
        \includegraphics[width=\textwidth]{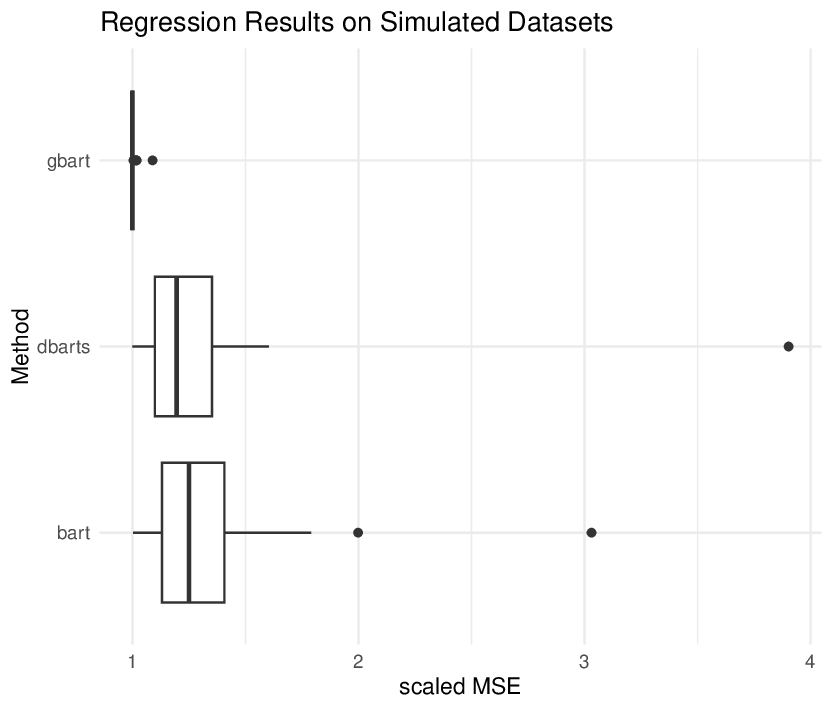}
    \end{minipage}
    \caption{Classification (left) and Regression (right) performance on simulated datasets: boxplots for MSE (regression) and CE (classification) loss observed across $30$ simulated datasets each for regression and classification respectively.}
    \label{fig:simulation_boxplots}
\end{figure}

We compare the performance of \texttt{GBART} with the R packages \texttt{BART} and \texttt{dbarts} on three simulation settings each for Gaussian regression and binary classification with `probit' link. In the regression setting, we simulate $Y \sim N(f_0(X), 1)$, and in the classification setting, we simulate $Y \sim Bin(1,\Phi(f_0(X)))$ where $\Phi$ denotes the probit link function and the true regression function $f_0(X)$ is one of the three Friedman functions \citep{friedman1991multivariate, friedman1993} described below.

\begin{align}
\textbf{Friedman 1:} \quad f_1(\mathbf{x}) &= 10\sin(\pi x_1 x_2) + 20\left(x_3 - 0.5\right)^2 + 10x_4 + 5x_5 \label{eq:friedman1} \\[6pt]
\textbf{Friedman 2:} \quad f_2(\mathbf{x}) &= \sqrt{x_1^2 + \left(x_2 x_3 - \frac{1}{x_2 x_4}\right)^2} \label{eq:friedman2} \\[6pt]
\textbf{Friedman 3:} \quad f_3(\mathbf{x}) &= \arctan\left(\frac{x_2 x_3 - \dfrac{1}{x_2 x_4}}{x_1}\right) \label{eq:friedman3}
\end{align}

The Friedman functions have been widely used for benchmarking tree-based regression models including BART \citep{chipman2010bart, linero2025generalized}. For each setting we simulate 10 random datasets, each with $500$ training and $500$ test samples. We simulate multivariate normal covariates $X \sim N_{20}(0, 0.1\mathbf{I})$, where $\mathbf{I}$ is the identity matrix. Of these 20 covariates, only the first five influence $Y$ for Friedman 1 and the first four for Friedman functions 2 and 3; the remainder are irrelevant. For Friedman 2 and Friedman 3, covariates are scaled to $[0,1]$ prior to applying the function to avoid numerical instability arising from near-zero denominators.

For each of the simulated datasets we train the suitable BART model (regression or classification) with $10{,}000$ posterior draws after discarding $5{,}000$ burn-in samples from a single Markov Chain Monte Carlo (MCMC) chain. We evaluate predictions on the test data and compare the three packages with respect to mean squared error (MSE) for regression and cross-entropy loss (CE) for classification on the test set. As is customary in the BART literature \citep{chipman2010bart}, we report scaled MSE and scaled CE, defined as the MSE or CE of each package divided by the minimum across the three packages for that dataset, so that a value of 1 indicates the best-performing package. Results across the three Friedman functions are summarized in Figure~\ref{fig:simulation_boxplots}, and results disaggregated by function are reported in Tables~\ref{tab: regression_simulation_friedman_scaled_mse} and~\ref{tab:classification_simulation_friedman_scaled_ce}. We observe that \texttt{GBART} performs comparably to the R packages \texttt{BART} (denoted \texttt{BART.R} in the figures and tables) and \texttt{dbarts} for both regression and classification.

% Figure 2: Real data boxplots side by side
\begin{figure}[htbp]
    \centering
    \begin{minipage}{0.48\textwidth}
        \centering
        \includegraphics[width=\textwidth]{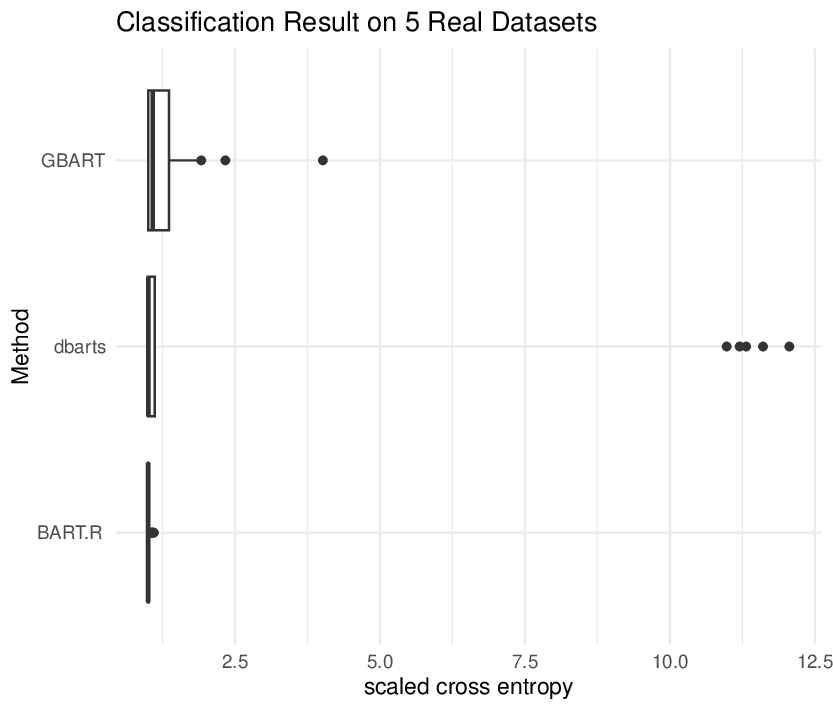}
    \end{minipage}
    \hfill
    \begin{minipage}{0.48\textwidth}
        \centering
        \includegraphics[width=\textwidth]{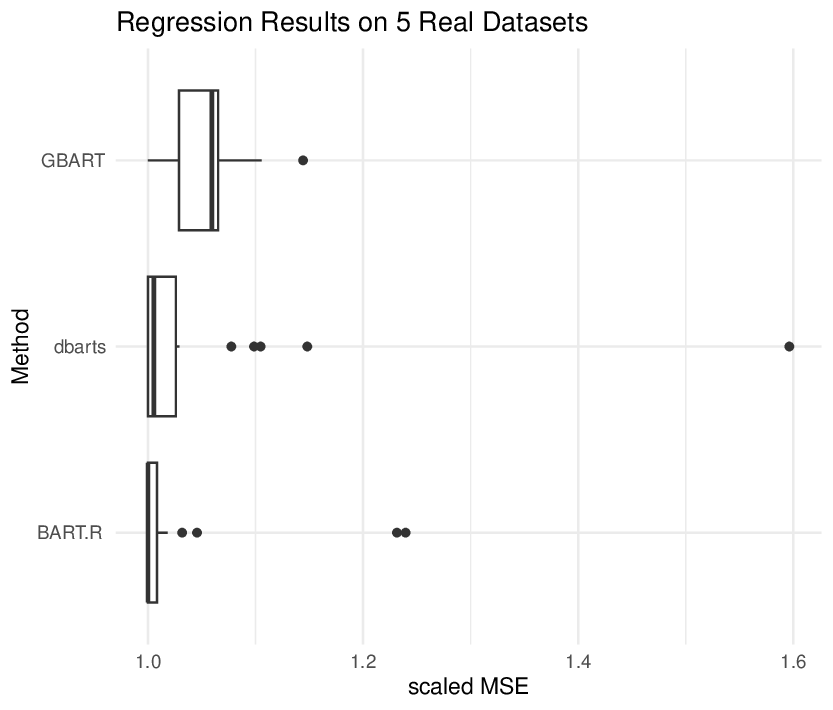}
    \end{minipage}
    \caption{Classification (left) and Regression (right) performance on real data: boxplots for MSE (regression) and CE (classification) loss observed across $5$ real datasets each for regression and classification respectively}
    \label{fig:realdata_boxplots}
\end{figure}

Next, we compare \texttt{GBART} with the existing R packages on 10 real datasets, 5 each for regression and classification. We repeat 10-fold cross-validation for 5 independent random splits of each dataset and evaluate MSE for regression and CE for classification on the held-out data for each random split. On the real datasets as well, \texttt{GBART} performs comparably to the R packages \texttt{BART} (denoted \texttt{BART.R} in the figures and tables) and \texttt{dbarts}. Results aggregated across all datasets are summarized in Figure~\ref{fig:realdata_boxplots}, and results disaggregated by dataset, along with data sources and sample sizes, are reported in Tables~\ref{tab:regression_datasets} and~\ref{tab:classification_datasets}.

\begin{figure}
    \centering
    \includegraphics[width=0.95\textwidth]{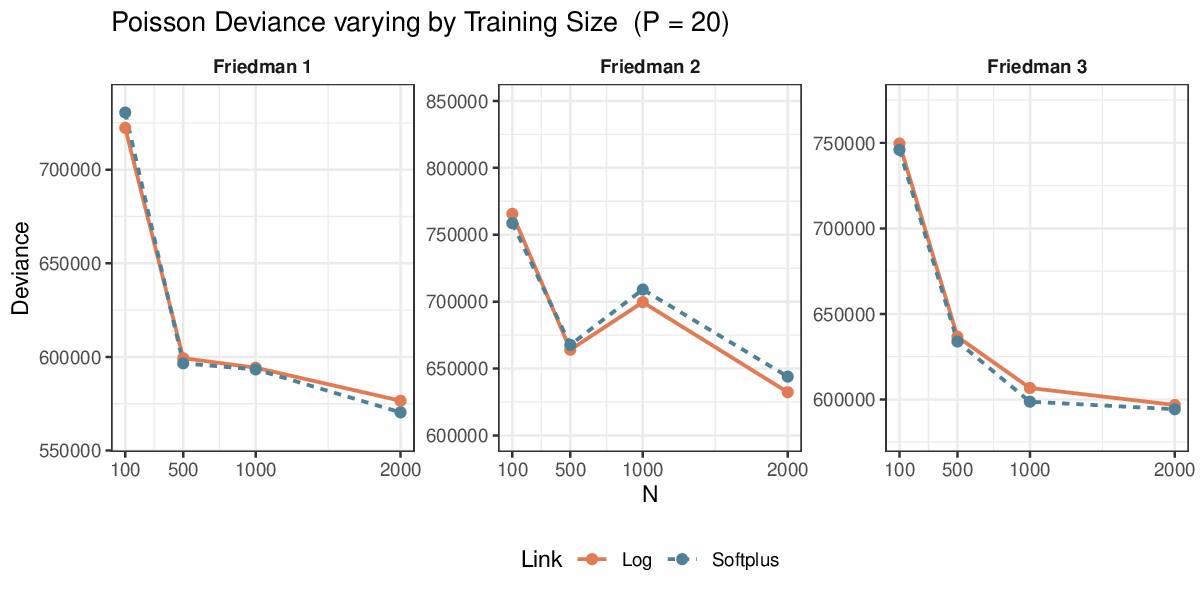}
    \vspace{0.5em}
    \includegraphics[width=0.95\textwidth]{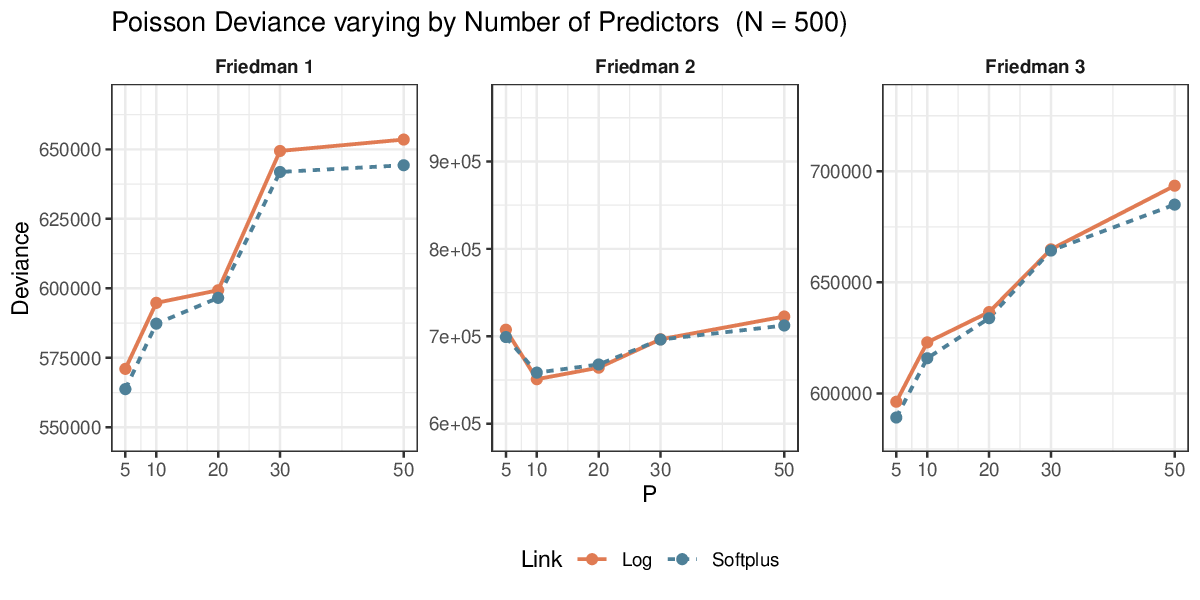}
    \caption{Poisson deviance computed on a test set of size $500$, using GBART with two different link functions: log and softplus. Top: Median deviance vs training data size $N$ with covariate dimension $P = 20$.
             Bottom: Median deviance vs $P$ with $N = 500$.}
    \label{fig:deviance_links}
\end{figure}

We apply \texttt{GBART} for regression on count data using two link functions: the canonical `log' link and the `softplus' link discussed in section \ref{sec: implications}. We simulate $Y \sim Poisson(|f_0(X)|)$, where $|\cdot|$ denotes the absolute value and the true regression function $f_0(X)$ is one of the three Friedman functions described in \eqref{eq:friedman1}-\eqref{eq:friedman3}. We compare across varying training sample sizes and varying number of covariates to evaluate whether any of the link functions perform better than the other. To compare how performance varies across number of training samples, we fix number of covariates to $p = 20$ and simulate ten random datasets consisting of $N \in \{100, 500, 1000, 2000\}$ training and $500$ test samples for each of the three Friedman functions (Figure \ref{fig:deviance_links} top). To compare how performance varies across number of covariates, we simulate ten random datasets consisting of $500$ training and $500$ test samples for each of the three Friedman functions and vary the number of covariates across $p \in \{10, 20, 50\}$ (Figure \ref{fig:deviance_links} bottom). For each setting we simulate $2000$ MCMC samples after $1000$ burn-in period.
 As sample size increases, deviance decreases for both link functions and as dimensions of covariates increase, deviance also increases for both. We did not observe any significanct difference in performance between the `log' and `softplus' links (p-value of Wilcoxon rank-sum test is $0.587$). Detailed simulation results are provided in Table \ref{tab:scaled_deviance}.
 
\section{Discussion}\label{sec:discussion}
In this paper we have examined a general framework for Bayesian Additive Regression Tree Models that encapsulates various conventional BART models adapted to a wide range of regression and classification tasks. We demonstrated that these models have a near-minimax posterior concentration rate for step functions, coordinatewise monotone functions, and H\"{o}lder continuous functions, thus corroborating the empirical success of BART and its variants, from a theoretical perspective. These results also build the foundation for uncertainty quantification statements for a wide variety of BART models, opening up an interesting avenue for future research; in particular, extending the semiparametric Bernstein-von Mises theorem of \cite{rockova2020semi} to the G-BART setting would be a natural next step. 
\iffalse
Among empirical implications, we have established the need for careful modeling choices such as selecting appropriate link functions. The theoretical results also substantiate the scope of a wider variety of distributions on approximating step-heights, which may prove advantageous for applications where the response distribution has a thicker tail.
\fi
Finally, to facilitate adoption of G-BART in practice, we have developed the \texttt{GBART} package, built on \texttt{PyMC-BART}, which implements G-BART for a range of exponential family distributions with a user-friendly interface modelled on R's \texttt{glm} function. An interesting direction for future work is the empirical investigation of non-Gaussian step-height distributions, such as the Laplace distribution, which our theoretical results show should also possess near-minimax concentration rates. However, implementing these within the Bayesian backfitting framework would require modifications to the sampling algorithm and is beyond the scope of the current package. We plan to explore this in future versions of \texttt{GBART}, alongside computational improvements as more efficient \texttt{PyMC-BART} implementations become available.

\section{Data Availability Statement}
Only publicly available real datasets are used in the paper. Software is available at \url{https://github.com/Enakshi-Saha/gbart}.

\section{Declaration of Interest}
The authors do not have any competing interests.

%% if your bibliography is in bibtex format, uncomment commands:
\bibliographystyle{apalike} % Style BST file (imsart-number.bst or imsart-nameyear.bst)
\bibliography{references}       % Bibliography file (usually '*.bib')

@article{diaconis1986consistency,
  title={On the consistency of Bayes estimates},
  author={Diaconis, Persi and Freedman, David},
  journal={The Annals of Statistics},
  pages={1--26},
  year={1986},
  publisher={JSTOR}
}

@article{ghosal2007convergence,
  title={Convergence rates of posterior distributions for noniid observations},
  author={Ghosal, Subhashis and Van Der Vaart, Aad and others},
  journal={The Annals of Statistics},
  volume={35},
  number={1},
  pages={192--223},
  year={2007},
  publisher={Institute of Mathematical Statistics}
}

@article{castillo2014bernstein,
  title={On the Bernstein--von Mises phenomenon for nonparametric Bayes procedures},
  author={Castillo, Isma{\"e}l and Nickl, Richard and others},
  journal={The Annals of Statistics},
  volume={42},
  number={5},
  pages={1941--1969},
  year={2014},
  publisher={Institute of Mathematical Statistics}
}

@article{cox1993analysis,
  title={An analysis of Bayesian inference for nonparametric regression},
  author={Cox, Dennis D},
  journal={The Annals of Statistics},
  pages={903--923},
  year={1993},
  publisher={JSTOR}
}

@article{murray2017log,
  title={Log-linear Bayesian additive regression trees for multinomial logistic and count regression models},
  author={Murray, Jared S},
  journal={Journal of the American Statistical Association},
  volume={116},
  number={534},
  pages={756--769},
  year={2021},
  publisher={Taylor \& Francis}
}

@article{chipman1998bayesian,
  title={Bayesian {CART} model search},
  author={Chipman, Hugh A and George, Edward I and McCulloch, Robert E},
  journal={Journal of the American Statistical Association},
  volume={93},
  number={443},
  pages={935--948},
  year={1998},
  publisher={Taylor \& Francis}
}

@inproceedings{rockova2019theory,
  title={On Theory for {BART}},
  author={Rockova, Veronika and Saha, Enakshi},
  booktitle={The 22nd International Conference on Artificial Intelligence and Statistics},
  pages={2839--2848},
  year={2019}
}

@article{stone1982optimal,
  title={Optimal global rates of convergence for nonparametric regression},
  author={Stone, Charles J},
  journal={The annals of statistics},
  pages={1040--1053},
  year={1982},
  publisher={JSTOR}
}

@inproceedings{he2019xbart,
  title={{XBART}: Accelerated {Bayesian} Additive Regression Trees},
  author={He, Jingyu and Yalov, Saar and Hahn, P Richard},
  booktitle={The 22nd International Conference on Artificial Intelligence and Statistics},
  pages={1130--1138},
  year={2019}
}

@article{pratola2019heteroscedastic,
  title={Heteroscedastic {BART} via multiplicative regression trees},
  author={Pratola, MT and Chipman, HA and George, EI and McCulloch, RE},
  journal={Journal of Computational and Graphical Statistics},
  pages={1--13},
  year={2019},
  publisher={Taylor \& Francis}
}

@article{hill2013assessing,
  title={Assessing lack of common support in causal inference using {Bayesian} nonparametrics: Implications for evaluating the effect of breastfeeding on children's cognitive outcomes},
  author={Hill, Jennifer and Su, Yu-Sung},
  journal={The Annals of Applied Statistics},
  pages={1386--1420},
  year={2013},
  publisher={JSTOR}
}

@article{linero2018bayesian,
  title={Bayesian regression trees for high-dimensional prediction and variable selection},
  author={Linero, Antonio R},
  journal={Journal of the American Statistical Association},
  volume={113},
  number={522},
  pages={626--636},
  year={2018},
  publisher={Taylor \& Francis}
}

@inproceedings{du2019interaction,
  title={Interaction detection with {Bayesian} decision tree ensembles},
  author={Du, Junliang and Linero, Antonio R},
  booktitle={The 22nd International Conference on Artificial Intelligence and Statistics},
  pages={108--117},
  year={2019},
  organization={PMLR}
}

@article{deshpande2020vc,
  title={{VC-BART}: {Bayesian} trees for varying coefficients},
  author={Deshpande, Sameer K and Bai, Ray and Balocchi, Cecilia and Starling, Jennifer E},
  journal={arXiv preprint arXiv:2003.06416},
  year={2020}
}

@article{pratola2014parallel,
  title={Parallel {Bayesian} additive regression trees},
  author={Pratola, Matthew T and Chipman, Hugh A and Gattiker, James R and Higdon, David M and McCulloch, Robert and Rust, William N},
  journal={Journal of Computational and Graphical Statistics},
  volume={23},
  number={3},
  pages={830--852},
  year={2014},
  publisher={Taylor \& Francis}
}

@misc{sparapani2019bart,
  title={The {BART} {R} package},
  author={Sparapani, Rodney and Spanbauer, Charles and McCulloch, Robert},
  year={2019}
}

@article{liu2020variable,
  title={Variable Selection via Thompson Sampling},
  author={Liu, Yi and Rockova, Veronika},
  journal={arXiv preprint arXiv:2007.00187},
  year={2020}
}

@article{rockova2020posterior,
  title={Posterior concentration for {Bayesian} regression trees and forests},
  author={Rockova, Veronika and van der Pas, St{\'e}phanie and others},
  journal={Annals of Statistics},
  volume={48},
  number={4},
  pages={2108--2131},
  year={2020},
  publisher={Institute of Mathematical Statistics}
}

@article{jeong2020art,
  title={The art of BART: Minimax optimality over nonhomogeneous smoothness in high dimension},
  author={Jeong, Seonghyun and Rockova, Veronika},
  journal={Journal of Machine Learning Research},
  volume={24},
  number={337},
  pages={1--65},
  year={2023}
}

@article{castillo2019multiscale,
  title={Uncertainty quantification for Bayesian CART},
  author={Castillo, Ismael and Rockov{\'a}, Veronika},
  journal={The Annals of Statistics},
  volume={49},
  number={6},
  pages={3482--3509},
  year={2021},
  publisher={Institute of Mathematical Statistics}
}

@article{dorie2025package,
  title={Package ‘dbarts’},
  author={Dorie, Vincent and Chipman, Hugh and McCulloch, Robert and Dadgar, Armon and Team, R Core and Draheim, Guido U and Bosmans, Maarten and Tournayre, Christophe and Petch, Michael and de Lucena Valle, Rafael and others},
  journal={},
  year={2025}
}

@article{li2024dynamic,
  title={Dynamic treatment regimes using bayesian additive regression trees for censored outcomes},
  author={Li, Xiao and Logan, Brent R and Hossain, SM Ferdous and Moodie, Erica EM},
  journal={Lifetime Data Analysis},
  volume={30},
  number={1},
  pages={181--212},
  year={2024},
  publisher={Springer}
}

@inproceedings{rockova2020semi,
  title={On Semi-parametric Inference for {BART}},
  author={Rockova, Veronika},
  booktitle={International Conference on Machine Learning},
  pages={8137--8146},
  year={2020},
  organization={PMLR}
}

@book{stein2009real,
  title={Real analysis: measure theory, integration, and Hilbert spaces},
  author={Stein, Elias M and Shakarchi, Rami},
  year={2009},
  publisher={Princeton University Press}
}

@article{quiroga2022bayesian,
  title={Bayesian additive regression trees for probabilistic programming},
  author={Quiroga, Miriana and Garay, Pablo G and Alonso, Juan M and Loyola, Juan Martin and Martin, Osvaldo A},
  journal={arXiv preprint arXiv:2206.03619},
  year={2022}
}

@article{gao2017minimax,
  title={ON ESTIMATION OF ISOTONIC PIECEWISE CONSTANT SIGNALS},
  author={Gao, Chao and Han, Fang and Zhang, Cun-Hui},
  journal={The Annals of Statistics},
  volume={48},
  number={2},
  pages={629--654},
  year={2020},
  publisher={JSTOR}
}

@article{biau2003risk,
  title={On the risk of estimates for block decreasing densities},
  author={Biau, G{\'e}rard and Devroye, Luc},
  journal={Journal of multivariate analysis},
  volume={86},
  number={1},
  pages={143--165},
  year={2003},
  publisher={Elsevier}
}

@article{liu2018abc,
  title={ABC Variable Selection with {Bayesian} Forests},
  author={Liu, Yi and Rockova, Veronika and Wang, Yuexi},
  journal={arXiv preprint arXiv:1806.02304},
  year={2018}
}

@article{denison1998bayesian,
  title={A {Bayesian} {CART} algorithm},
  author={Denison, David GT and Mallick, Bani K and Smith, Adrian FM},
  journal={Biometrika},
  volume={85},
  number={2},
  pages={363--377},
  year={1998},
  publisher={Oxford University Press}
}

@article{sparapani2016nonparametric,
  title={Nonparametric survival analysis using {Bayesian} additive regression trees ({BART})},
  author={Sparapani, Rodney A and Logan, Brent R and McCulloch, Robert E and Laud, Purushottam W},
  journal={Statistics in medicine},
  volume={35},
  number={16},
  pages={2741--2753},
  year={2016},
  publisher={Wiley Online Library}
}

@inproceedings{van2017bayesian,
  title={Bayesian dyadic trees and histograms for regression},
  author={van der Pas, St{\'e}phanie and Rockova, Veronika},
  booktitle={Advances in neural information processing systems},
  pages={2089--2099},
  year={2017}
}

@article{starling2020bart,
  title={Bart with targeted smoothing: An analysis of patient-specific stillbirth risk},
  author={Starling, Jennifer E and Murray, Jared S and Carvalho, Carlos M and Bukowski, Radek K and Scott, James G and others},
  journal={Annals of Applied Statistics},
  volume={14},
  number={1},
  pages={28--50},
  year={2020},
  publisher={Institute of Mathematical Statistics}
}

@article{nelder1972generalized,
  title={Generalized linear models},
  author={Nelder, John Ashworth and Wedderburn, Robert WM},
  journal={Journal of the Royal Statistical Society: Series A (General)},
  volume={135},
  number={3},
  pages={370--384},
  year={1972},
  publisher={Wiley Online Library}
}

@article{chipman2010bart,
  title={BART: Bayesian additive regression trees},
  author={Chipman, Hugh A and George, Edward I and McCulloch, Robert E and others},
  journal={The Annals of Applied Statistics},
  volume={4},
  number={1},
  pages={266--298},
  year={2010},
  publisher={Institute of Mathematical Statistics}
}

@article{bleich2014variable,
  title={Variable selection for BART: an application to gene regulation},
  author={Bleich, Justin and Kapelner, Adam and George, Edward I and Jensen, Shane T},
  journal={The Annals of Applied Statistics},
  pages={1750--1781},
  year={2014},
  publisher={JSTOR}
}

@article{hill2011bayesian,
  title={Bayesian nonparametric modeling for causal inference},
  author={Hill, Jennifer L},
  journal={Journal of Computational and Graphical Statistics},
  volume={20},
  number={1},
  pages={217--240},
  year={2011},
  publisher={Taylor \& Francis}
}

@article{hahn2017bayesian,
  title={Bayesian regression tree models for causal inference: regularization, confounding, and heterogeneous effects},
  author={Hahn, P Richard and Murray, Jared S and Carvalho, Carlos},
  journal={arXiv preprint arXiv:1706.09523},
  year={2017}
}

@article{taddy2011dynamic,
  title={Dynamic trees for learning and design},
  author={Taddy, Matthew A and Gramacy, Robert B and Polson, Nicholas G},
  journal={Journal of the American Statistical Association},
  volume={106},
  number={493},
  pages={109--123},
  year={2011},
  publisher={Taylor \& Francis}
}

@article{chipman2016high,
  title={mBART: multidimensional monotone BART},
  author={Chipman, Hugh A and George, Edward I and McCulloch, Robert E and Shively, Thomas S},
  journal={Bayesian Analysis},
  volume={17},
  number={2},
  pages={515--544},
  year={2022},
  publisher={International Society for Bayesian Analysis}
}

@article{linero2017bayesian,
  title={Bayesian regression tree ensembles that adapt to smoothness and sparsity},
  author={Linero, Antonio Ricardo and Yang, Yun},
  journal={arXiv preprint arXiv:1707.09461},
  year={2017}
}

@article{linero2025generalized,
  title={Generalized Bayesian additive regression trees models: Beyond conditional conjugacy},
  author={Linero, Antonio R},
  journal={Journal of the American Statistical Association},
  volume={120},
  number={549},
  pages={356--369},
  year={2025},
  publisher={Taylor \& Francis}
}

@article{friedman1991multivariate,
  title={Multivariate adaptive regression splines},
  author={Friedman, Jerome H},
  journal={The annals of statistics},
  volume={19},
  number={1},
  pages={1--67},
  year={1991},
  publisher={Institute of Mathematical Statistics}
}

@techreport{friedman1993,
  author      = {Friedman, Jerome H.},
  title       = {Fast {MARS}},
  institution = {Department of Statistics, Stanford University},
  year        = {1993},
  number      = {110}
}

@book{tsybakov2009introduction,
  title     = {Introduction to Nonparametric Estimation},
  author    = {Tsybakov, Alexandre B.},
  year      = {2009},
  publisher = {Springer},
  address   = {New York},
  series    = {Springer Series in Statistics},
  doi       = {10.1007/b13794}
}

@article{basak2022semiparametric,
  title={Semiparametric analysis of clustered interval-censored survival data using soft Bayesian additive regression trees (SBART)},
  author={Basak, Piyali and Linero, Antonio and Sinha, Debajyoti and Lipsitz, Stuart},
  journal={Biometrics},
  volume={78},
  number={3},
  pages={880--893},
  year={2022},
  publisher={Wiley Online Library}
}

@article{um2023bayesian,
  title={Bayesian additive regression trees for multivariate skewed responses},
  author={Um, Seungha and Linero, Antonio R and Sinha, Debajyoti and Bandyopadhyay, Dipankar},
  journal={Statistics in medicine},
  volume={42},
  number={3},
  pages={246--263},
  year={2023},
  publisher={Wiley Online Library}
}

\newpage
\appendix
\bigskip
\begin{center}
{\large\bf SUPPLEMENTARY MATERIALS}
\end{center}
\renewcommand{\thetable}{S\arabic{table}}
\setcounter{table}{0}

\section{Bayesian CART Prior by Denison et al. (1998)}\label{denison_prior}
We describe the Bayesian CART prior proposed by \cite{denison1998bayesian}. The prior on individual Bayesian trees is assigned conditional on the number of terminal nodes/ leaves $K$ and all prior probability is concentrated on the set of all {\it valid} tree partitions, as defined below (Definition 3.1 of \cite{rockova2020posterior}):
\begin{definition}
Denote by $\bOmega = \{\Omega\}_{k=1}^K$, a partition of $[0,1]^p$, We say that $\bOmega$ is valid if
\begin{equation}
    \mu(\Omega_k) \ge \frac{C}{n} \quad \forall k = 1, \dots, K
\end{equation}
for some $C \in \mathbb{N}\setminus\{0\}$.
\end{definition}
Valid partitions have non-empty cells, where the allocation does not need to be balanced. Now the prior on tree partitions is specified as follows:
\begin{enumerate}
\item The number of leaves in a tree $K$ follows a Poisson distribution with parameter $\lambda > 0$
\begin{equation}
P(K)=\frac{\lambda^K}{(e^\lambda-1)K!}, \qquad k=1,2,\dots
\end{equation}
\item Given the number of leaves $K$, a tree is chosen uniformly at random from the set of all available {\it valid} tree-partitions with $K$ leaves. Number of valid tree partitions is given by
\begin{equation}
\Delta(V_K)=\frac{q^{K-1}n!}{(n-K+1)!}
\end{equation}
This is a slightly modified version of the original prior proposed by \cite{denison1998bayesian}. This modified version was used by \cite{rockova2020posterior} to derive posterior concentration rates for the BART estimator under this prior.
\item At each node, the splitting rule consists of picking a split variable $j$ uniformly at random from the available directions $\{1,\dots, q\}$ and picking a split point $c$ uniformly at random from the data values $\{x_{ij} : \X_i \in \Omega\}$ falling within the current cell $\Omega$.
\end{enumerate}

\section{Preliminary Results with Proof}
\begin{lemma}
The multivariate Gaussian $\mathcal{N}_p ({\bf 0}, \mathbb{I}_p)$ and the multivariate Laplace $\mathcal{L}_p ({\bf 0}, \mathbb{I}_p)$ distribution belong to the general family of distributions with CDF $F_\beta$ that has the following property:
For some $C_1>0$, $0 < C_2 \le 2$ and $C_3 > 0$ and any $t>0$,
    \begin{align}
        F_\beta(\norm{\beta}_\infty \le t) \gtrsim \left (e^{-C_1 t^{C_2}} t \right)^{p} \quad \text{for $t>0$}
    \end{align}
    \vspace{-10pt}
    \begin{align}
        F_\beta(\norm{\beta}_\infty \ge t) \lesssim e^{-C_3 t} \quad \text{for $t \ge 1$}
    \end{align}
\end{lemma}
\begin{proof}
If $F_\beta= \mathcal{N}_p ({\bf 0}, \mathbb{I}_p)$, then for any $t>0$,
\begin{equation*}
    F_\beta(\norm{\beta}_\infty \le t) \gtrsim \left (e^{- t^2/2} \int_{-t}^t d\beta \right )^p \gtrsim e^{-p t^2 / 2}t^p
\end{equation*}
For $t \ge 1$
\begin{equation*}
    F_\beta(\norm{\beta}_\infty \ge t) \lesssim  \left (e^{-t^2/4} 2\int_{t}^\infty e^{-z^2/4} dz \right )^p \lesssim e^{-C_3t}
\end{equation*}
If $F_\beta = \mathcal{L}_p ({\bf 0}, \mathbb{I}_p)$, then for any $t > 0$,
\begin{equation*}
    F_\beta(\norm{\beta}_\infty \le t) \gtrsim \left (e^{- t} \int_{-t}^t d\beta \right )^p \gtrsim e^{-p t}t^p
\end{equation*}
Also, for any $t \ge 0$,
\begin{equation*}
    F_\beta(\norm{\beta}_\infty \ge t)
    = 1 - \left(1 - e^{-t}\right)^p
    \le p\, e^{-t}
    \lesssim e^{-t},
\end{equation*}
using Taylor series expansion around $0$.
\end{proof}
\begin{lemma}
Let $f$ and $f_0$ denote step functions of the form $f(\X)=\sum_{k=1}^K\beta_k\mathbb{I}(\X\in\Omega_k)$ and $f_0(\X)=\sum_{k=1}^K\beta^0_k\mathbb{I}(\X\in\Omega_k)$ respectively, on a tree-shaped partition $\{\Omega_k\}_{k=1}^K$. Let $P_{f}$ and $P_{f_0}$ denote two probability densities belonging to an Exponential family distribution of the form 
\begin{equation}
    P_f(\Y \C \X)=h(\Y)g\left[f(\X)\right]\exp \left [\eta \left (f(\X) \right)^T T(\Y) \right],
\end{equation}
with parameters $f$ and $f_0$ respectively. If $\abs{\frac{\nabla^T g (\b)}{g(\b)}} \le C^n_{g} \textbf{1}_p$, for some positive sequence $\{C^n_{g}\}_{n \ge 1}$, then
\begin{align}\label{eq:KL_bound}
& K_n(P_{f_0}, P_f) \vee V_n(P_{f_0}, P_f)  \lesssim C^n_g \sum_{k=1}^K \norm{\beta_k-\beta^0_k}_1
\end{align}
\begin{align}\label{eq:Hellinger_bound}
& H_n^2(P_{f_0}, P_f)  \lesssim C^n_g \sum_{k=1}^K \norm{\beta_k-\beta^0_k}_1
\end{align}
\end{lemma}
\begin{proof}
Denoting $f_i=f(\X_i)$ and $f_{i0}=f_0(\X_i)$, and recalling that under Assumption~1 we have $\eta(z)=z$, we bound $K_n(P_{f_0}, P_f)$ as follows. We write
\begin{align*}
K_n(P_{f_0},P_f)
&= \frac{1}{n}\sum_{i=1}^n \int P_{f_0}(\Y \C \X_i)\log\frac{P_{f_0}(\Y \C \X_i)}{P_f(\Y \C \X_i)}\,\d\Y \\
&= \frac{1}{n}\sum_{i=1}^n \int h(\Y)g(f_{i0})\exp\!\left[f_{i0}^T T(\Y)\right]
   \!\left[\log\frac{g(f_{i0})}{g(f_i)} + (f_{i0}-f_i)^T T(\Y)\right]\d\Y \\
&= \frac{1}{n}\sum_{i=1}^n \!\left[\log\frac{g(f_{i0})}{g(f_i)}
   + (f_{i0}-f_i)^T \mathbb{E}_{P_{f_{i0}}}\!\left[T(\Y)\right]\right].
\end{align*}
For the exponential family \eqref{eq:exponential_pdf} with $\eta(z)=z$, we can show % differentiating the normalization identity $g(f)\int h(\Y)\exp[f^T T(\Y)]\d\Y = 1$ with respect to $f$ yields the score identity 
$\mathbb{E}_{P_f}[T(\Y)] = -\nabla \log g(f) = -\nabla g(f)/g(f)$. Substituting this to the above equation gives:
\begin{align*}
K_n(P_{f_0},P_f)
&= \sum_{k=1}^K \mu(\Omega_k)
   \!\left[\log\frac{g(\beta^0_k)}{g(\beta_k)}
   - \frac{\nabla^T g(\beta^0_k)}{g(\beta^0_k)}(\beta^0_k-\beta_k)\right].
\end{align*}
Performing a first-order Taylor expansion of $\log g(\beta^0_k)$ around $\beta_k$ gives $\log g(\beta^0_k) - \log g(\beta_k) \approx \frac{\nabla^T g(\beta_k)}{g(\beta_k)}(\beta^0_k - \beta_k)$. Combining with the triangle inequality:
\begin{align*}
K_n(P_{f_0},P_f) \lesssim \sup_\beta \abs{\frac{\nabla^T g(\b)}{g(\b)}}\sum_{k=1}^K\norm{\beta_k-\beta^0_k}_1
 = C^n_g \sum_{k=1}^K\norm{\beta_k-\beta^0_k}_1.
\end{align*}
The same argument gives an identical bound for $V_n(P_{f_0},P_f)$.

Since the squared Hellinger distance satisfies $H^2(P,Q) \le K(P\|Q)$ (see e.g. Lemma 2.4 of \cite{tsybakov2009introduction}), we obtain
\begin{equation*}
H_n^2(P_{f_0}, P_f) \lesssim K_n(P_{f_0},P_f) \lesssim C^n_g \sum_{k=1}^K \norm{\beta_k-\beta^0_k}_1.
\end{equation*}
\end{proof}

\iffalse
\begin{lemma}
Any bounded coordinatewise {\bf monotone} function $f_0$ can be approximated with arbitrary precision $\varepsilon_n$, by a step function supported on a k-d tree partition with $\wh{K} \ge \lceil 1/\epsilon_n \rceil$ leaves.
\end{lemma}

Without loss of generality, assume $0 \le f_0(\cdot) \le 1$. Partition interval $[0,1]$ by $0 = y_0 < y_1 < \dots < y_k< \dots < y_{K-1} < y_K = 1$, with $K=\lceil 1/\epsilon_n \rceil$. Then $\C y_k-y_{k-1} \C < \varepsilon_n$ and we can approximate $f_0(\X)$ by the step function:
$$ f(\X)= \sum_{k=1}^K y_k \mathbb{I}\{\X \in \Omega_k\}$$,
where $\Omega_k=f^{-1}[y_{k-1},y_k]$.

Even though $\Omega_k$ is not a box and may have any arbitrary shape, we can partition $\Omega_k$ into a fine-enough axis-paralleled grid and define a new step function that approximates $f$ on $\Omega_k$, with arbitrary accuracy. Thus, any complex, arbitrarily partitioned step function, such as $f(\cdot)$, can be approximated by a step function defined on a finer, axis-parallel grid.

Furthermore, since any step function supported on an axis-paralleled partition has an equivalent step function supported on a k-d tree, we can approximate the step function $ f(\X)= \sum_{k=1}^K y_k \mathbb{I}\{\X \in \Omega_k\}$ by a step function supported on a recursive binary tree partition $\{\wh{\Omega}_k\}_{k=1}^{\wh{K}}$ with number of leaves $\wh{K} \ge K$.
\fi
\paragraph{Proof of Lemma \ref{lemma:monotone_approx}}
\begin{proof}
Without loss of generality, assume $0 \le f_0(\cdot) \le 1$ and that $f_0$ is 
coordinatewise non-decreasing (replacing $x_j$ with $1-x_j$ for any 
decreasing coordinate). Partition the range $[0,1]$ by 
$0=y_0 < y_1 < \dots < y_K=1$ with $K=\lceil 1/\varepsilon_n \rceil$, 
so $|y_k - y_{k-1}| < \varepsilon_n$. Define the axis-aligned boxes
$$
B_k = \prod_{j=1}^q \left[\inf\{x_j : f_0(\x) \in [y_{k-1},y_k)\},\; 
                          \sup\{x_j : f_0(\x) \in [y_{k-1},y_k)\}\right],
$$
and the step function $f(\X) = \sum_{k=1}^K y_k \mathbb{I}\{\X \in B_k\}$. 
Since $f_0$ is coordinatewise non-decreasing, any $\x$ with 
$f_0(\x) \in [y_{k-1}, y_k)$ satisfies $\x \in B_k$ by construction, 
and hence $|f_0(\X) - f(\X)| \le \varepsilon_n$ for all $\X$. The 
$\{B_k\}$ are axis-aligned boxes, so the partition can be refined into 
a k-d tree partition $\{\wh\Omega_k\}_{k=1}^{\wh K}$ with 
$\wh K \ge K = \lceil 1/\varepsilon_n \rceil$ leaves.
\end{proof}

\section{Proof of Main Results}\label{sec:result_proof}
In this section we prove Theorem 4.1 and Theorem 4.3. Most steps in the proofs are identical and hence for simplicity we describe the common steps of the proofs together and mark the steps that are different by the corresponding theorem number. We need to prove three conditions: entropy condition (C1), prior concentration condition (C2) and prior decay rate condition (C3). The steps of the proofs for each of these conditions are described below.

\subsection{Entropy Condition}\label{sub:entropy_proof}
Define the sieve $\mF_n \subset \mF$ as follows,
\begin{equation*}
    \mF_n = \{f_{\mT,{\b}}(\X) \text{, with $\mT$ being any tree partition with number of leaves} K \le k_n \text{ and } \norm{\beta}_\infty \le C^n_\beta\},
\end{equation*}
where $k_n \propto n \varepsilon_n^2 / \log n$ and $C^n_\beta$ is defined in Assumption 1.

Since $\norm{\bm{z}}_1 \le Kp\norm{\bm{z}}_\infty$ for any $\bm{z} \in \mathbb{R}^{Kp}$, by the bound \eqref{eq:Hellinger_bound} and by definition of $\mF_n$, we can write
\begin{align*}
 N \left(\frac{\varepsilon}{36}, \mF_n \cap \mathcal{A}_\varepsilon, H_n\right)
 & \lesssim \sum_{K=1}^{k_n} \Delta(V_K) \cdot N \left(\frac{\varepsilon^2}{36C^n_g Kp},
 \{\beta: \norm{\beta}_\infty \le C^n_\beta\}, \norm{\cdot}_\infty \right) \\
 & \lesssim \sum_{K=1}^{k_n} (qn)^K \cdot 
 \left (\frac{36 C_\beta^n C_g^n Kq}{\varepsilon^2} \right )^{Kq}
\end{align*}
Therefore the logarithm of the LHS of (C1) can be bounded from above by
\begin{align*}
 (k_n+1)p \left [\log 36 + \log (C_\beta^n C_g^n) + \log k_n + \log (pq) - 2\log \varepsilon_n \right ]
\end{align*}
Since $C_\beta^n C_g^n \lesssim n^M$ for some $M>0$, ignoring smaller terms, proving condition (C1) reduces to proving
\begin{align}\label{eq:entropy_final_exponential}
 (k_n+1)p \log n \lesssim n\varepsilon_n^2
\end{align}

\paragraph{Theorem 4.1:} When $f_0$ is a step function with complexity $K_{f_0}$ we can prove \eqref{eq:entropy_final_exponential} by replacing $\varepsilon_n= n^{-1/2}\sqrt{K_{f_0}\log^{2\eta}(n/ K_{f_0})}$ and $k_n \propto \frac{n\varepsilon_n^2}{p \log (n/  K_{f_0})}= K_{f_0}\log^{2\theta-1}(n/K_{f_0})$ for some $\theta > 1/2$.

\paragraph{Theorem 4.3:} When $f_{0l}$ is a $\nu$-H\"{o}lder continuous function with $0 < \nu \le 1$ for all $l=1,\dots,p$, replacing $\varepsilon_n= n^{-\nu/(2\nu+q)}\sqrt{\log n}$ and $k_n \propto \frac{n\varepsilon_n^2}{\log n}= n^{q/(2\nu+q)}$ proves \eqref{eq:entropy_final_exponential}.

\subsection{Prior Concentration Condition} \label{sub:concentration_proof}
Let $\wt{f}_0=\left (f_{\mT,{\B_1^0}}(\x), \dots, f_{\mT,{\B_{q-1}^0}}(\x) \right )$ denote the projection of $f_0$ onto a balanced k-d tree partition with $a_n$ leaves, where $a_n$ is chosen so that $\norm{f_0-\wt{f}_0}_{2,n} < \varepsilon_n/2$. 

A k-d tree partition is a balanced full binary tree that partitions $[0,1]^q$ into nearly equal-sized rectangular cells by cycling over coordinate directions $\{1,\dots,q\}$ and splitting each internal node at the median of the observations along the current axis, so that after $s$ rounds of splits all $K=2^{qs}$ terminal nodes contain approximately $n/K$ observations \citep{rockova2019theory}.

\paragraph{Theorem 4.1:} If $f_0$ is a step function, $a_n=K_{f_0}$

\paragraph{Theorem 4.3:} If $f_0$ is a $\nu$-H\"{o}lder continuous function, $a_n$ is chosen by the following lemma, which is analogous to Lemma 3.2 of \cite{rockova2020posterior}.

\begin{lemma}
Denote $f=\{f_l\}_{l=1}^{p}$ and assume $f_l \in \mathcal{H}^{\nu_l}$ where $\nu_l \le 1$ for all $l=1,\dots,p$ and $\mathcal{X}$ is regular. Then there exists tree structured step functions $\hat{f}=\{f_{\mT,\B_l}\}_{l=1}^{p} \in \mF_K$ for some given tree partition $\mT$ with $K \in \mathbb{N}$ leaves such that for some constant $C > 0$,
$$
\norm{\hat{f}-f}_{2,n} \le C q \sum_{l=1}^{p} \left ( \frac{1}{K^{\nu_l/q}}  \norm{f_l}_{\mathcal{H}^{\nu_l}}\right) \le C \frac{q}{K^{\nu/q}}  \sum_{l=1}^{p} \left ( \norm{f_l}_{\mathcal{H}^{\nu_l}}\right),
$$
where $\nu=\min_{l=1}^{p}\nu_l$.
\end{lemma}
As a corollary, replacing $C_0=C \left (\sum_{l=1}^{p} \norm{f_l}_{\mathcal{H}^\nu}\right)$, $a_n$ satisfies
\begin{equation}\label{eq:a_n_bound}
\left (\frac{2C_0 q}{\varepsilon_n} \right)^{q/\nu} \le a_n \le \left (\frac{2C_0 q}{\varepsilon_n} \right)^{q/\nu} + 1
\end{equation}
Using \eqref{eq:KL_bound} and by triangle inequality, we can bound the LHS of (C2) from below by
\begin{align*}
    C \pi(a_n) \Pi \left (\beta \in B_n^{a_n} : \norm{\beta-\beta^0}_1 \le \frac{\epsilon_n^2}{2C_g^n} \right)
\end{align*}
For the prior by \cite{chipman2010bart}, $C=1$ and $\pi(a_n) \gtrsim e^{-a_n \log a_n}$ (by Corollary 5.2 of \cite{rockova2019theory}).

For the prior by \cite{denison1998bayesian},
$C=\frac{1}{\C F_{a_n}\C} > (a_n d n)^{-a_n} > e^{-a_n \log a_n}$  (by Lemma 3.1 of \cite{rockova2020posterior}) and $\pi(a_n) \gtrsim e^{-a_n \log a_n}$ (by proof of Theorem 4.1 of \cite{rockova2020posterior}).

Thus for both priors $C\pi(a_n) \gtrsim e^{-2a_n\log a_n}$.

Next we bound $\Pi \left (\beta \in B_n^{a_n}: \norm{\b-\b^0}_1 \le \frac{\epsilon_n^2}{2C_g^n} \right)$, up to a constant, from below by
\begin{align*}
 \Pi \left (\b: \norm{\b}_\infty \le C_\beta^n, \quad \norm{\b-\b^0}_\infty \le \frac{\epsilon_n^2}{2a_n q C_g^n} \right)
\end{align*}
Since $C^n_g$ and $C^n_\beta$ both are increasing with $n$, for sufficiently large $n$, the above expression is bounded below by
\begin{align*}
& \Pi \left (\b: \norm{\b-\b^0}_\infty \le \frac{\varepsilon_n^2}{2a_n p C_g^n} \right) \\
\gtrsim & e^{- C_1 a_n p \left (\norm{\beta_0}_\infty+ \frac{\varepsilon_n^2}{2a_n p C_g^n} \right )^{C_2}}\left ( {\norm{\beta_0}_\infty +  \frac{\varepsilon_n^2}{2a_n p C_g^n}} \right)^{a_n p}
\end{align*}
Since $\varepsilon_n^2 \rightarrow 0$ and both $a_n$ and $C_g^n$ are both increasing with $n$, assuming $\norm{f_0}_\infty \lesssim \sqrt{\log n}$, the above bound reduces to $e^{- C_1 a_n p \log^{C_2/2} n + \frac{a_n p}{2}\log\log n}$.

We can prove $e^{-a_n \log n} \gtrsim e^{-n\varepsilon_n^2}$ for Theorem 4.1 and Theorem 4.3 separately by replacing appropriate values of $\varepsilon_n$.  Since $C_2 \le 2$, this would complete the proof. 
\subsection{Prior Decay Rate Condition} \label{sub:decay_proof}

\paragraph{Theorem 4.1:} When $f_0$ is a step-function with complexity $K_{f_0}$,
\begin{align*}
\Pi(\mF \setminus \mF_n)  & \le \Pi(\mF \setminus \bigcup_{K=1}^{k_n} F_K)  + \Pi(\bigcup_{K \le k_n} \{f \in F_K: \norm{\beta}_\infty > C_\beta^n \}) \\
& \le \Pi(\bigcup_{K>k_n}F_K)  + e^{-K_{f_0}\log n/2}\\
  & = \Pi(\bigcup_{K>k_n}F_K)  + o(e^{-n \varepsilon_n^2})
\end{align*}
The last line is due to the fact $C_\beta^n \gtrsim K_{f_0}\log n$ when $f_0$ is a step-function with complexity $K_{f_0}$.

\paragraph{Theorem 4.3:} When $f_0$ is a $\nu$-H\"{o}lder continuous function, the LHS of condition (C3) can be bounded from above by $\Pi(\bigcup_{K>k_n}F_K)  + o(e^{-n \varepsilon_n^2})$, as follows:

\begin{align*}
  \Pi(\mF \setminus \mF_n)  & \le \Pi(\mF \setminus \bigcup_{K=1}^{k_n} F_K)  + \Pi(\bigcup_{K \le k_n} \{f \in F_K: \norm{\beta}_\infty > C_\beta^n \}) \\
  & \le \Pi(\bigcup_{K>k_n}F_K)  + \sum_{K=1}^{k_n}\Pi(\{\beta: \norm{\beta}_\infty > C_\beta^n\}) \\
  \end{align*}
  By condition \eqref{eq:beta_tail2}, the right hand side can be bounded by
  \begin{align*}
  \Pi(\bigcup_{K>k_n}F_K)  + \sum_{K=1}^{k_n} e^{-C^n_\beta}
  \le \Pi(\bigcup_{K>k_n}F_K)  + k_n e^{-C^n_\beta}\\
  \end{align*}
Since $C_\beta^n \gtrsim n$, when $f_0$ is a $\nu$-H\"{o}lder continuous functions, we have
$k_n e^{-C^n_\beta} = o(e^{-n \varepsilon_n^2})$.
Thus, it is enough to show that
\begin{equation*}
\Pi(\bigcup_{K>k_n}F_K) \lesssim e^{-n\varepsilon_n^2}
\end{equation*}
This condition is satisfied for both the priors under consideration. This follows from section 8.3 of \cite{rockova2020posterior} for the prior by \cite{denison1998bayesian} and from Corollary 5.2 of \cite{rockova2019theory} for the prior by \cite{chipman2010bart}.

\section{Classification with Dirichlet Step Heights}\label{sec:result_dirichlet}
For a multi-class classification problem with $p$ classes, where the response variable $\Y$ is a categorical random variable with $p$ categories, $\Y$ can be written as a $p$ dimensional binary vector that has $1$ at the $l$-th coordinate if $\Y$ belongs to category $l \in \{1,\dots,p\}$ and $0$ elsewhere. G-BART assumes
\begin{equation}
    \Y \C \X \sim \text{Multinomial}(1, \bm{f_0}(\X)),
\end{equation}
where $\bm{f_0} = \left (f_{01},\dots,f_{0p} \right )':\mathbb{R}^q \rightarrow (0,1)^p $ is a constrained function with $\bm{f_0}(\X)'\textbf{1}_p = 1$ for any $\X \in \mathbb{R}^q$. Each $f_{0l}(\cdot)$ can be approximated by a step function of the form
\begin{equation}\label{eq:tree_mapping_dirichlet}
f_{\mT,{P}}(\x)=\sum_{k=1}^K P_k\mathbb{I}(\x\in\Omega_k)
\end{equation}
on a tree-partition $\{\Omega_k\}_{k=1}^K$. \cite{denison1998bayesian} assumes
\begin{equation}\label{eq:dirichlet_dist}
    P_k = \left (P_{k1},\dots,P_{kp} \right ) \stackrel{i.i.d}{\sim} \text{Dirichlet}(\alpha_1,\dots,\alpha_p),
\end{equation}
where $\alpha_l > 0, \quad \forall l \in \{1,\dots,p\}$.

\begin{theorem}\label{thm:dirichlet_tree}
 If we assume that the distribution of the step-sizes satisfies \eqref{eq:dirichlet_dist}, then under Assumptions 1 \& 2 described in section 4 of the manuscript, the Bayesian Tree estimator satisfies the following property,: 
 \begin{enumerate}
 \item[(i)] If $f_0$ is $\nu$-H\"{o}lder continuous with  {$0<\nu\leq 1$} where $\|f_0\|_\infty\lesssim \log^{1/2} n$, then with $\varepsilon_n=n^{-\alpha/(2\alpha+p)}\log^{1/2} n$, and
 \item[(ii)] If $f_0$ is step-function with complexity $K_{f_0} \lesssim \sqrt{n}$, 
 
 then with $\varepsilon_n=n^{-1/2}\sqrt{K_{f_0}p \log^{2\nu}{\left(n/K_{f_0}p\right)}} n$, 
 \end{enumerate}
\begin{equation*}
\Pi\left( f \in \mathcal{F}: H_n(\P_f,\P_{f_0}) > M_n\,\varepsilon_n \mid \Y^{(n)}\right) \to 0,
\end{equation*}
for any $M_n \to \infty$  in $\P_{f_0}^{(n)}$-probability, as $n,p \to \infty$. 
 
The above statement is true for both tree priors considered in this paper: the prior by \cite{denison1998bayesian} and a modified version of the prior by \cite{chipman1998bayesian} with $p_{split}(\Omega_t)=\alpha^{d(\Omega_t)}$ for some $1/n\leq \alpha<1/2$.
\end{theorem}
\begin{proof}
We need to prove three conditions: entropy condition (C1), prior concentration condition (C2) and prior decay rate condition (C3). Among these (C1) and (C3) can be proved by the same technique as in section \ref{sec:result_proof}. Therefore we will only prove Condition (C3) here. We need to show, for some $c>0$
\begin{equation}\label{dirichlet_prior_bound}
\Pi \left (f \in \mathcal{F}: max \{K_n(f,f_0),V_n(f,f_0) \} \le \varepsilon_n^2 \right) \gtrsim e^{-c n \varepsilon_n^2}
\end{equation}
Let $\wt{f}_0 = \left (f_{\mT,{P_1^0}}(\x), \dots, f_{\mT,{P_q^0}}(\x) \right )$ denote the projection of $f_0$ onto a balanced k-d tree partition $\mT$ with $a_n$ leaves, where $a_n$ is chosen so that $\norm{f_0-\wt{f}_0}_{2,n} < \varepsilon_n/2$. If $f_0$ is a step function, $a_n=K_{f_0}$. If $f_0$ is a $\nu$-H\"{o}lder continuous function, $a_n$ is chosen by Lemma 3.2 of \cite{rockova2020posterior}, where replacing $C_0=C \left (\sum_{l=1}^{p} \norm{f_l}_{\mathcal{H}^\nu}\right)$ we get
\begin{equation}
\left (\frac{2C_0 q}{\varepsilon_n} \right)^{q/\nu} \le a_n \le \left (\frac{2C_0 q}{\varepsilon_n} \right)^{q/\nu} + 1
\end{equation}
$f_{\mT,{P^0_l}}(\x)$ is of the form \eqref{eq:tree_mapping_dirichlet} for some tree topology $\mT$ with $a_n$ leaves and $P^0_l=\{P_{kl}^0\}_{k=1}^{a_n}$ for $l=1,\dots,p$. We assume there exists some $\delta_0 > 0$ such that $\min{f_{0l}} > \delta_0$ for all $l=1,\dots,q$. This implies $P_{lk}^0 > \delta_0$ for all $l=1,\dots q$ and all $k=1,\dots,K$. Therefore by \eqref{eq:KL_bound}, we can bound the LHS of \eqref{dirichlet_prior_bound} from above by
\begin{align*}
C \pi(a_n) \Pi \left (P \in [0,1]^{a_np}: \norm{P-P^0}_1 \le \delta_0 \varepsilon_n^2/2 \right) 
\end{align*}
For the prior by \cite{chipman1998bayesian}, $C=1$  and for the prior by \cite{denison1998bayesian}, $C=\frac{1}{\C F_{a_n}\C} > (a_n d n)^{-a_n} > e^{-a_n \log a_n}$  (by Lemma 3.1 of \cite{rockova2020posterior}). By Corollary 5.2 of \cite{rockova2019theory} for the prior by \cite{chipman1998bayesian} and by proof of Theorem 4.1 of \cite{rockova2020posterior} for the prior by \cite{denison1998bayesian}, we can show $\pi(a_n) \ge e^{-a_n \log a_n}$. Thus for both priors,
\begin{equation}\label{eq:dirichlet_pi_bound}
    C\pi(a_n) > e^{-2a_n\log a_n}
\end{equation}
Since $P_k \sim \text{Dirichlet}(\alpha_1,\dots,\alpha_p)$ for all $k=1,\dots,K$ and $P_{lk}^0 > \delta_0$, for all $l=1,\dots,p$ and all $k=1,\dots,K$, we can bound  $\Pi \left (P \in [0,1]^{a_nq}: \norm{P-P^0}_1 \le \delta_0 \epsilon_n^2/2 \right)$ from above by
\begin{align}\label{eq:dirichlet_p_bound}
\Pi \left (P \in [0,1]^{a_np}: \norm{P-P^0}_\infty \le \frac{\delta_0 \epsilon_n^2}{2a_n p} \right) \gtrsim C_{\alpha} (\frac{\delta_0 \epsilon_n^2}{a_n p})^{a_n p},
\end{align}
where $C_{\alpha}$ is a constant that depends on the Dirichlet parameters $\alpha=(\alpha_1,\dots,\alpha_q)$. Combining \eqref{eq:dirichlet_pi_bound} and \eqref{eq:dirichlet_p_bound} completes the proof.
\end{proof}

\section{Generalized BART is Resistant to Overfitting}\label{parsimony} 
As a by-product of the theoretical results discussed in section \ref{sec:results}, we can the following statements which support the empirical observation that generalized Bayesian trees are resilient to overfitting.
\begin{itemize}
\item[(i)] Under the assumptions of Theorem \ref{theorem: step}  we have 
$
\Pi\left(K \gtrsim K_{f_0} \mid \Y^{(n)}\right) \to 0
$
in $\P_{f_0}^{(n)}$-probability, as $n,q \to \infty$.

\item[(ii)] Under the assumptions of Theorem \ref{theorem: monotone} we have 
$
\Pi\left(K \gtrsim n^{q/(2+q)} \mid \Y^{(n)}\right) \to 0
$
in $\P_{f_0}^{(n)}$-probability, as $n,q \to \infty$.

\item[(iii)] Under the assumptions of Theorem \ref{theorem: smooth}  we have 
$
\Pi\left(K \gtrsim n^{q/(2\nu+q)} \mid \Y^{(n)}\right) \to 0
$
in $\P_{f_0}^{(n)}$-probability, as $n,q \to \infty$.
\end{itemize}
%\end{corollary}
\begin{proof}
The proofs of (i), (ii) and (iii) follow from Lemma 1 of \cite{ghosal2007convergence}, in conjunction with the proofs of Theorems \ref{theorem: step}, \ref{theorem: monotone} and \ref{theorem: smooth} respectively.
\end{proof}

\section{Additional Tables}\label{sec: supplemtal tables} 

\begin{table}[ht]
\centering
\caption{Regression: Package Performance Comparison on Simulated Data: Median Scaled MSE (Q1, Q3)}
\label{tab: regression_simulation_friedman_scaled_mse}
\begin{tabular}{@{}llll@{}}
\toprule
Example    & gbart                & bart                 & dbarts               \\ \midrule
Friedman 1 & 1.000 (1.000, 1.000) & 1.244 (1.190, 1.311) & 1.170 (1.150, 1.208) \\
Friedman 2 & 1.000 (1.000, 1.012) & 1.109 (1.050, 1.461) & 1.081 (1.000, 1.387) \\
Friedman 3 & 1.000 (1.000, 1.000) & 1.357 (1.249, 1.456) & 1.290 (1.211, 1.386) \\ \bottomrule
\end{tabular}
\end{table}

\begin{table}[ht]
\centering
\caption{Classification: Package Performance Comparison on Simulated Data: Median Scaled CE (Q1, Q3)}
\label{tab:classification_simulation_friedman_scaled_ce}
\begin{tabular}{@{}llll@{}}
\toprule
Example    & gbart                & bart                 & dbarts               \\ \midrule
Friedman 1 & 1.002 (1.000, 1.006) & 1.000 (1.000, 1.001) & 1.177 (1.131, 1.195) \\
Friedman 2 & 1.008 (1.004, 1.008) & 1.000 (1.000, 1.000) & 1.141 (1.112, 1.179) \\
Friedman 3 & 1.001 (1.000, 1.004) & 1.000 (1.000, 1.002) & 1.172 (1.139, 1.188) \\ \bottomrule
\end{tabular}
\end{table}

\iffalse
\begin{table}[ht]
\centering
\caption{Package Performance Comparison on Simulated Data: Median Scaled CE (Q1, Q3)}
\label{tab:scaled_metrics}
\begin{tabular}{@{}llll@{}}
\toprule
Task & gbart & bart & dbarts \\ \midrule
Regression (Scaled MSE)    & 1.000 (1.000, 1.000) & 1.250 (1.131, 1.407) & 1.196 (1.099, 1.352) \\
Classification (Scaled CE) & 1.004 (1.000, 1.007) & 1.000 (1.000, 1.001) & 1.171 (1.124, 1.188) \\ \bottomrule
\end{tabular}
\end{table}
\fi

\begin{table}[ht]
\centering
\caption{Computation Time Comparison on Simulated Data (seconds): Median (Q1, Q3)}
\label{tab:compute_time}
\begin{tabular}{@{}llll@{}}
\toprule
Task           & gbart                      & bart                 & dbarts              \\ \midrule
Regression     & 1125.25 (1098.05, 1133.47) & 47.69 (47.33, 48.00) & 35.24 (35.06, 35.42) \\
Classification & 1620.22 (1609.57, 1635.40) & 53.70 (53.39, 53.90) & 35.89 (35.43, 36.13) \\ \bottomrule
\end{tabular}
\end{table}

\begin{table}
\centering
\caption{Scaled Median MSE (Q1, Q3) in Regression: Package Performance Comparison on Real Data.}
\label{tab:regression_datasets}
\begin{tabular}{@{}llllll@{}}
\toprule
Data       & $n$  & Source  & GBART                & BART-R               & dbarts               \\ \midrule
abalone    & 4177 & UCI     & 1.038 (1.036, 1.042) & 1.001 (1.000, 1.002) & 1.000 (1.000, 1.005) \\
bloodbrain & 96   & \texttt{caret}   & 1.065 (1.060, 1.100) & 1.002 (1.000, 1.015) & 1.000 (1.000, 1.029) \\
uscrime    & 47   & \texttt{MASS}    & 1.061 (1.000, 1.091) & 1.046 (1.000, 1.231) & 1.099 (1.078, 1.148) \\
wine red   & 1599 & UCI     & 1.019 (1.014, 1.026) & 1.000 (1.000, 1.008) & 1.009 (1.000, 1.022) \\
wine white & 4898 & UCI     & 1.063 (1.060, 1.065) & 1.000 (1.000, 1.000) & 1.003 (1.002, 1.010) \\ \bottomrule
\end{tabular}
\end{table}

\begin{table}
\centering
\caption{Scaled Median CE (Q1, Q3) in Classification: Package Performance Comparison on Real Data.}
\label{tab:classification_datasets}
\begin{tabular}{@{}llllll@{}}
\toprule
Data        & $n$  & Source & GBART                & BART-R               & dbarts                 \\ \midrule
biopsy       & 683  & \texttt{MASS}   & 1.886 (1.866, 2.335) & 1.000 (1.000, 1.000) & 11.314 (11.203, 11.606) \\
birthwt      & 189  & \texttt{MASS}   & 1.070 (1.070, 1.362) & 1.004 (1.003, 1.008) & 1.000 (1.000, 1.000)   \\
gallstone    & 320    & UCI    & 1.000 (1.000, 1.000) & 1.073 (1.062, 1.093) & 1.095 (1.081, 1.119)   \\
germancredit & 1000 & \texttt{caret}  & 1.072 (1.012, 1.083) & 1.001 (1.000, 1.004) & 1.000 (1.000, 1.002)   \\
segmentation & 2079 & \texttt{caret}  & 1.327 (1.324, 1.331) & 1.001 (1.000, 1.007) & 1.004 (1.000, 1.007)   \\ \bottomrule
\end{tabular}
\end{table}

\begin{table}
\centering
\caption{Scaled median deviance (Q1, Q3) by link function. Values are deviance divided by the minimum deviance between log and softplus for each run; the better link scores 1.000.}
\label{tab:scaled_deviance}
\begin{tabular}{cclcc}
\toprule
$P$ & $N$ & Example & Log & Softplus \\
\midrule
\multirow{3}{*}{5} & \multirow{3}{*}{500}
  & Friedman 1 & 1.012 (1.006, 1.017) & 1.000 (1.000, 1.000) \\
& & Friedman 2 & 1.000 (1.000, 1.004) & 1.015 (1.000, 1.037) \\
& & Friedman 3 & 1.010 (1.007, 1.014) & 1.000 (1.000, 1.000) \\
\midrule
\multirow{3}{*}{10} & \multirow{3}{*}{500}
  & Friedman 1 & 1.011 (1.004, 1.013) & 1.000 (1.000, 1.000) \\
& & Friedman 2 & 1.000 (1.000, 1.001) & 1.007 (1.001, 1.013) \\
& & Friedman 3 & 1.007 (1.003, 1.011) & 1.000 (1.000, 1.000) \\
\midrule
\multirow{12}{*}{20}
  & \multirow{3}{*}{100}
    & Friedman 1 & 1.000 (1.000, 1.002) & 1.009 (1.001, 1.020) \\
& & Friedman 2 & 1.006 (1.000, 1.010) & 1.000 (1.000, 1.040) \\
& & Friedman 3 & 1.013 (1.005, 1.029) & 1.000 (1.000, 1.000) \\
\cmidrule{2-5}
  & \multirow{3}{*}{500}
    & Friedman 1 & 1.004 (1.000, 1.012) & 1.000 (1.000, 1.002) \\
& & Friedman 2 & 1.000 (1.000, 1.004) & 1.005 (1.000, 1.021) \\
& & Friedman 3 & 1.008 (1.006, 1.013) & 1.000 (1.000, 1.000) \\
\cmidrule{2-5}
  & \multirow{3}{*}{1000}
    & Friedman 1 & 1.008 (1.004, 1.014) & 1.000 (1.000, 1.000) \\
& & Friedman 2 & 1.000 (1.000, 1.001) & 1.002 (1.000, 1.006) \\
& & Friedman 3 & 1.007 (1.003, 1.011) & 1.000 (1.000, 1.000) \\
\cmidrule{2-5}
  & \multirow{3}{*}{2000}
    & Friedman 1 & 1.010 (1.006, 1.013) & 1.000 (1.000, 1.000) \\
& & Friedman 2 & 1.000 (1.000, 1.005) & 1.017 (1.002, 1.019) \\
& & Friedman 3 & 1.010 (1.008, 1.012) & 1.000 (1.000, 1.000) \\
\midrule
\multirow{3}{*}{30} & \multirow{3}{*}{500}
  & Friedman 1 & 1.008 (1.004, 1.014) & 1.000 (1.000, 1.000) \\
& & Friedman 2 & 1.000 (1.000, 1.005) & 1.004 (1.000, 1.008) \\
& & Friedman 3 & 1.003 (1.000, 1.006) & 1.000 (1.000, 1.001) \\
\midrule
\multirow{3}{*}{50} & \multirow{3}{*}{500}
  & Friedman 1 & 1.013 (1.006, 1.015) & 1.000 (1.000, 1.000) \\
& & Friedman 2 & 1.008 (1.001, 1.018) & 1.000 (1.000, 1.000) \\
& & Friedman 3 & 1.013 (1.010, 1.017) & 1.000 (1.000, 1.000) \\
\bottomrule
\end{tabular}
\end{table}

\iffalse
\begin{table}
\centering
\caption{Package Performance Comparison on Real Data: Median Scaled Loss (Q1, Q3)}
\label{tab:bart_results}
\begin{tabular}{llll}
\toprule
Task           & GBART (R)                & BART.R               & dbarts               \\
\midrule
Classification & 1.083 (1.004, 1.362) & 1.003 (1.000, 1.009) & 1.007 (1.000, 1.119) \\
Regression     & 1.059 (1.029, 1.065) & 1.000 (1.000, 1.008) & 1.005 (1.000, 1.026) \\
\bottomrule
\end{tabular}
\end{table}
\fi

\end{document}